\documentclass{amsart}

\usepackage[main=english,french]{babel}
\usepackage{amscd,amsmath,amssymb,amsfonts,bbm, calligra, xspace}
\usepackage[T1]{fontenc}
\usepackage{lmodern}
\usepackage{mathtools}
\usepackage{enumerate, enumitem}
\usepackage{multicol}
\usepackage[all]{xy}
\usepackage{mathrsfs}
\usepackage{footmisc}
\usepackage{graphicx}
\usepackage{microtype}
\usepackage{tikz-cd}
\usepackage{quiver}
\usetikzlibrary{babel}
\usepackage{euscript}
\usepackage{stackrel}

\usepackage{xcolor}
\definecolor{darkblue}{rgb}{0,0,0.8}
\definecolor{darkgreen}{rgb}{0,0.4,0}
\usepackage[colorlinks=true,linkcolor=darkblue, citecolor=darkgreen,  urlcolor=darkgreen, unicode,pdfborder={0 0 0},final]{hyperref}

\newtheorem{thm}{Theorem}[section]
\newtheorem{prop}[thm]{Proposition}

\newtheorem{lem}[thm]{Lemma}
\newtheorem{cor}[thm]{Corollary}

\theoremstyle{definition}

\newtheorem{quest}[thm]{Question}
\newtheorem{quests}[thm]{Questions}

\theoremstyle{remark}
\newtheorem{rem}[thm]{Remark}

\newtheorem{rems}[thm]{Remarks}

\newtheorem{Step}{Step}

\numberwithin{equation}{section}

\newcommand{\ci}{\mathcal{C}^{\infty}}

\newcommand{\Zb}{\mathbb{Z}}
\newcommand{\Rb}{\mathbb{R}}

\newcommand{\Cb}{\mathbb{C}}
\newcommand{\Qb}{\mathbb{Q}}
\newcommand{\Gb}{\mathbb{G}}
\newcommand{\Lb}{\mathbb{L}}
\newcommand{\Pb}{\mathbb{P}}
\newcommand{\Abb}{\mathbb{A}}

\newcommand{\Oc}{\mathcal{O}}

\newcommand{\CH}{\mathrm{CH}}
\newcommand{\Hr}{\mathrm{H}}

\newcommand{\Kbf}{\mathbf{K}}
\newcommand{\Ibf}{\mathbf{I}}
\newcommand{\Ir}{\mathrm{I}}

\newcommand{\Wr}{\mathrm{W}}
\newcommand{\et}{\textup{\'et}}

\newcommand{\Wbf}{\mathbf{W}}

\newcommand{\Lc}{\mathcal{L}}

\newcommand{\Hbf}{\mathbf{H}}

\newcommand{\Hsc}{\mathscr{H}}

\newcommand{\Fc}{\mathcal{F}}

\newcommand{\nr}{\mathrm{nr}}

\newcommand{\Ec}{\mathcal{E}}

\newcommand{\ovarepsilon}{\overline{\varepsilon}}

\DeclareMathOperator{\im}{im}

\DeclareMathOperator{\Id}{Id}

\DeclareMathOperator{\Spec}{Spec}

\DeclareMathOperator{\colim}{colim}

\DeclareMathOperator{\ord}{ord}

\DeclareMathOperator{\pt}{pt}
\DeclareMathOperator{\sign}{sign}

\DeclareMathOperator{\Pic}{Pic}

\DeclareMathOperator{\Gal}{Gal}
\DeclareMathOperator{\ev}{ev}
\DeclareMathOperator{\oev}{\overline{ev}}

\DeclareMathOperator{\cl}{cl}

\renewcommand{\Im}{\operatorname{Im}}

\makeatletter 

\newcommand*\Neginternal[3]{\mathpalette\Neg@{{#1}{#2}{#3}}}
\newcommand*\Neg@[2]{\Neg@@{#1}#2}
\newcommand*\Neg@@[4]{%
  \mathrel{\ooalign{%
    $\m@th#1#4$\cr
    \hidewidth$\m@th#3{#1}\mkern\muexpr#2*2$\hidewidth\cr
  }}%
}

\newcommand*\negslash[1]{\m@th#1\not\mathrel{\phantom{=}}}
\newcommand*\snegslash[1]{\rotatebox[origin=c]{60}{$\m@th#1-$}}
\newcommand*\ssnegslash[1]{\rotatebox[origin=c]{60}{$\m@th#1{\dabar@}\mkern-7mu{\dabar@}$}}
\newcommand*\sssnegslash[1]{\rotatebox[origin=c]{60}{$\m@th#1\dabar@$}}
\makeatother  

\makeatletter
\DeclareFontFamily{OMX}{MnSymbolE}{}
\DeclareSymbolFont{MnLargeSymbols}{OMX}{MnSymbolE}{m}{n}
\SetSymbolFont{MnLargeSymbols}{bold}{OMX}{MnSymbolE}{b}{n}
\DeclareFontShape{OMX}{MnSymbolE}{m}{n}{
    <-6>  MnSymbolE5
   <6-7>  MnSymbolE6
   <7-8>  MnSymbolE7
   <8-9>  MnSymbolE8
   <9-10> MnSymbolE9
  <10-12> MnSymbolE10
  <12->   MnSymbolE12
}{}
\DeclareFontShape{OMX}{MnSymbolE}{b}{n}{
    <-6>  MnSymbolE-Bold5
   <6-7>  MnSymbolE-Bold6
   <7-8>  MnSymbolE-Bold7
   <8-9>  MnSymbolE-Bold8
   <9-10> MnSymbolE-Bold9
  <10-12> MnSymbolE-Bold10
  <12->   MnSymbolE-Bold12
}{}

\let\llangle\@undefined
\let\rrangle\@undefined
\DeclareMathDelimiter{\llangle}{\mathopen}%
                     {MnLargeSymbols}{'164}{MnLargeSymbols}{'164}
\DeclareMathDelimiter{\rrangle}{\mathclose}%
                     {MnLargeSymbols}{'171}{MnLargeSymbols}{'171}
\makeatother

\newcommand{\isoto}{\myxrightarrow{\,\sim\,}}
\makeatletter
\def\myrightarrow{{\setbox\z@\hbox{$\rightarrow$}\dimen0\ht\z@\multiply\dimen0 6\divide\dimen0 10\ht\z@\dimen0\box\z@}}
\def\myrightarrowfill@{\arrowfill@\relbar\relbar\myrightarrow}
\newcommand{\myxrightarrow}[2][]{\ext@arrow 0359\myrightarrowfill@{#1}{#2}}
\makeatother

\makeatletter
\newcommand{\extp}{\@ifnextchar^\@extp{\@extp^{\,}}}
\def\@extp^#1{\mathop{\bigwedge\nolimits^{\!#1}}}
\makeatother

\def\loccit{\emph{loc}.\kern3pt \emph{cit}.{}\ }
\def\eg{e.g.\kern.3em}

\DeclareRobustCommand{\resp}{\text{resp.}\kern.2em}

\begin{document}

\author{Olivier Benoist}
\address{DMA, École normale supérieure, Université PSL, CNRS, 75005 Paris, France}
\email{olivier.benoist@ens.fr}

\author{Samuel Lerbet}
\address{DMA, École normale supérieure, Université PSL, CNRS, 75005 Paris, France}
\email{samuel.lerbet@ens.psl.eu}

\begin{abstract}
Let $X$ be a smooth algebraic variety over $\Rb$. We study the separation of the connected components of $X(\Rb)$ by equivariant cohomology, unramified cohomology, and signatures of unramified quadratic forms. Our results highlight how the geometry of $X$ influences these questions. We apply them to cubic hypersurfaces, complete intersections of two quadrics, abelian varieties, and Enriques surfaces. We also obtain new bounds on the exponent of the cokernel of higher signature morphisms in $\Ibf$-cohomology and investigate the optimality of these bounds for smooth projective varieties.
\end{abstract}

\title[Separating components by cohomological invariants]{Separating components of real algebraic varieties by cohomological invariants}


\maketitle

\section{Introduction}

Let $X$ be a smooth algebraic variety of pure dimension $d$ over the field $\Rb$ of real numbers. Its real locus $X(\Rb)$ is a (possibly empty) $\ci$ manifold of dimension~$d$. It may have several connected components, even if $X$ is connected. 

Let $A(X)\subset\Rb(X)^{\times}$ be the multiplicative group of rational functions on $X$ that are regular and nonzero at all real points of $X$. Consider the morphism
\begin{equation}
\label{signmap}
A(X)\to \Hr^0(X(\Rb),\Zb/2)
\end{equation}
associating with $f\in A(X)$ and $x\in X(\Rb)$ the value $0$ if $f(x)>0$ and the value~$1$ if~$f(x)<0$. When $X$ is projective, the Stone--Weierstrass theorem implies that the morphism \eqref{signmap} is surjective (see e.g.\,\cite[Satz a)]{brockerReelleDivisoren1980}). 
In other words, one can prescribe the signs of an element of $A(X)$ on the connected components of~$X(\Rb)$.

In this article, we study similar questions, where $A(X)$ is replaced by equivariant cohomology groups (see \S\ref{parcoho}), unramified cohomology groups (see \S\ref{parunr}), or~$\Ibf$\nobreakdash-cohomology groups such as the unramified Witt group of $X$ (see \S\ref{parWitt}), and~\eqref{signmap} is replaced by evaluation at real points (in the first two cases) or signature morphisms (in the latter case). In other words, can one prescribe the values at real points of cohomology or unramified cohomology classes of~$X$? Or the signatures at real points of unramified quadratic forms over~$X$? 

We are primarily interested in the case where $X$ is projective, our goal being to understand the influence of the geometry of $X$ on these questions.

\subsection{Equivariant cohomology}
\label{parcoho}

Set $G:=\Gal(\Cb/\Rb)$. If $M$ is a $G$-module (or a bounded complex of $G$-modules), let~$\Hr^*_G(-,M)$ denote $G$-equivariant sheaf cohomology with values in $M$. We often take~${M=\Zb/2}$ (with trivial $G$-action) or~$M=\Zb(j)$ (the abelian group~$\Zb$ on which complex conjugation acts by~$(-1)^j$). If~$X$ is an algebraic variety over~$\Rb$, one can endow $X(\Cb)$ with a natural action of $G$. The groups $\Hr^*_G(X(\Cb),M)$ then play the same role, in real algebraic geometry, as the usual Betti cohomology groups in complex algebraic geometry (see e.g.\,\cite{krasnovEquivariantGrothendieckCohomology1994,vanHamelAlgebraicCyclesTopology2000,benoistIntegralHodgeConjecture2020}).

As $\Hr^j_G(\pt,\Zb/2)=\Zb/2$ (for $j\geq 0$) and~$\Hr^j_G(\pt,\Zb(j))=\Zb/2$ (for $j\geq 1$), evaluation at real points gives rise to morphisms
\begin{equation}
\label{evalcoho}
\begin{aligned}
\oev_X^j&:\Hr^j_G(X(\Cb),\Zb/2)\to \Hr^0(X(\Rb),\Zb/2)\quad\text{for }j\geq 0,\\
\text{and}\quad
\ev_X^j&:\Hr^j_G(X(\Cb),\Zb(j))\to \Hr^0(X(\Rb),\Zb/2)\quad\text{for }j\geq 1
\end{aligned}
\end{equation}
(see \S\ref{parrecollections}). 
We wish to understand when the morphisms \eqref{evalcoho} are surjective. Since the reduction mod $2$ morphisms $\Hr^j_G(X(\Cb),\Zb(j))\to \Hr^j_G(X(\Cb),\Zb/2)$ and the Bockstein morphisms $\Hr^j_G(X(\Cb),\Zb/2)\to\Hr^{j+1}_G(X(\Cb),\Zb(j+1))$ are compatible with evaluation maps, the surjectivity of $\ev_X^j$ implies the surjectivity of $\oev_X^j$, which in turn implies the surjectivity of $\ev_X^{j+1}$. We note that $\oev_X^j$ and $\ev_X^j$ are always surjective if~$X(\Rb)$ is connected, as the compositions ${\Hr^j_G(\pt,\Zb/2)\to \Hr^j_G(X(\Cb),\Zb/2)\xrightarrow{\oev_X^j} \Hr^0(X(\Rb),\Zb/2)}$ and $\Hr^j_G(\pt,\Zb(j))\to \Hr^j_G(X(\Cb),\Zb(j))\xrightarrow{\ev_X^j} \Hr^0(X(\Rb),\Zb/2)$
are surjective.

The following result is due to Krasnov \cite[Corollary 3.2]{krasnovEquivariantGrothendieckCohomology1994} in the projective case and to van Hamel \cite[Theorem 2.8]{vanHamelTorsionZeroCycles2000} in general (see also \cite[Remark~1.6\,(ii)]{benoistIntegralHodgeConjecture2020}).

\begin{thm}[Krasnov, van Hamel]
\label{thKvH}
Let $X$ be a smooth variety of dimension $d$ over $\Rb$. Then $\oev^d_X$ is surjective.
\end{thm}

Our first goal is to investigate the optimality of Theorem \ref{thKvH}.

\begin{thm}[see Propositions \ref{propno2torsion} and \ref{propva} and Corollaries \ref{corocubic} and \ref{corevd}]\,
\label{th1}
\begin{enumerate}[label=(\roman*)] 
\item 
\label{cohoi}
Let $X$ be a smooth variety of dimension $d\geq 1$ over $\Rb$. If $\Hr^{d+1}(X(\Cb),\Zb)[2]$ vanishes, then $\ev^d_X$ is surjective.
\item 
\label{cohoii}
For all $d\geq 1$, there exists a smooth projective variety $X$ of dimension $d$ over $\Rb$ such that $\Hr^{d+1}(X(\Cb),\Zb)[2]=0$ and $\oev^{d-1}_X$ is not surjective.
\item
\label{cohoiii}
For all $d\geq 2$, there exists a smooth projective variety $X$ of dimension $d$ over $\Rb$ such that $\ev^d_X$ is not surjective.
\end{enumerate}
\end{thm}

In~\ref{cohoii}, one can take $X$ to be a smooth cubic hypersurface, a smooth complete intersection of two quadrics or an abelian variety. The variety $X$ in \ref{cohoiii} can be chosen to be the product of an Enriques surface with~$d-2$ elliptic curves. In the case of Enriques surfaces, we find a necessary and sufficient condition for $\ev_X^2$ to be surjective (see Corollary~\ref{corEnriques}).

\subsection{Unramified cohomology}
\label{parunr}

Let $X$ be a smooth variety of dimension $d$ over~$\Rb$. Let $M$ be a $G$-module (or a bounded complex of $G$-modules). For $j\in\Zb$, let~$\Hsc^j_X(M)$ be the Zariski sheaf on~$X$ associated with the presheaf ${U\mapsto \Hr_G^j(U(\Cb),M)}$. The \emph{unramified cohomology groups} of~$X$ are the groups $\Hr^j_{\nr}(X,M):=\Hr^0(X,\Hsc^j_X(M))$ (see \cite{colliot-theleneRealComponentsAlgebraic1990,colliot-theleneZerocyclesCohomologyReal1996,vanHamelTorsionZeroCycles2000} and~\cite[\S5.1]{benoistIntegralHodgeConjecture2020}). Evaluation at real points again gives rise to morphisms
\begin{equation}
\begin{aligned}
\label{unreval}
\oev_{\nr,X}^j&:\Hr^j_{\nr}(X,\Zb/2)\to \Hr^0(X(\Rb),\Zb/2)\quad\text{for }j\geq 0,\\
\text{and}\quad
\ev_{\nr,X}^j&:\Hr^j_{\nr}(X,\Zb(j))\to \Hr^0(X(\Rb),\Zb/2)\quad\text{for }j\geq 1.
\end{aligned}
\end{equation}
Arguing as in \S\ref{parcoho} shows that the surjectivity of $\ev_{\nr,X}^j$ implies the surjectivity of~$\oev_{\nr,X}^j$, which in turn implies the surjectivity of $\ev_{\nr,X}^{j+1}$. In addition, making use of the natural morphisms $\Hr^j_G(X(\Cb),M)\to \Hr^j_{\nr}(X,M)$ shows that the surjectivity of~$\ev_X^j$ (\resp of $\oev_X^j$) implies the surjectivity of $\ev_{\nr,X}^j$ (\resp of $\oev_{\nr,X}^j$). 

Investigating the surjectivity of the morphisms \eqref{unreval} is harder than for \eqref{evalcoho}, since the unramified cohomology groups of $X$ depend not only on its topology but also on its cycle-theoretic properties. This line of research was initiated by Colliot-Th\'el\`ene and Parimala, who showed that $\oev^{j}_{\nr,X}$ is bijective for $j\geq d+1$ (see~\cite[Theorem~2.3.1]{colliot-theleneRealComponentsAlgebraic1990}). It follows from Theorem \ref{thKvH} that $\oev^d_{\nr,X}$ is surjective (answering a question raised in~\cite[Remark~2.4.4]{colliot-theleneRealComponentsAlgebraic1990}).

We contribute to this theme in two directions. First, we give geometric criteria depending only on the complex variety $X_{\Cb}$ for other evaluation maps ($\ev^d_{\nr,X}$, $\oev^{d-1}_{\nr,X}$, $\ev^{d-1}_{\nr,X}$ and $\oev^{d-2}_{\nr,X}$) to be surjective (see Corollary \ref{cortriangle2}). The images of $\oev^j_X$ and~$\oev_{\nr,X}^j$ (resp.\,of $\ev^j_X$ and $\ev_{\nr,X}^j$) are always equal when $j\leq 2$ (see Corollary~\ref{corpetitdegre}); our second contribution consists of criteria for these images to be equal when~$j=3$ (see Theorem~\ref{th3bis}). This can be used to obstruct the surjectivity of~$\oev_{\nr,X}^3$ (resp.\,of~$\ev_{\nr,X}^3$). 

The following theorem illustrates the applicability of these results.

\begin{thm}[Corollary \ref{corod=23} and Propositions \ref{propciquadrics}, \ref{propcubic} and \ref{propva2}]\,
\label{th4}
\begin{enumerate}[label=(\roman*)] 
\item 
\label{exai}
For $d\in\{2,3\}$, there exists a smooth projective variety $X$ of dimension $d$ over~$\Rb$ such that $\ev^d_{\nr,X}$ is not surjective.
\item 
\label{exaii}
Let $X$ be a smooth complete intersection of two quadrics of dimension $d$ over~$\Rb$ such that $X(\Rb)$ is not connected. If $d\in \{1,2,3\}$, then $\oev^{d-1}_{\nr,X}$ is not surjective. If~$d\geq 4$, then $\ev^{d-1}_{\nr,X}$ is surjective.
\item
\label{exaiii}
Let $X$ be a smooth cubic hypersurface of dimension $d$ over~$\Rb$ such that~$X(\Rb)$ is not connected. If $d\in\{1,2,3\}$, or if $d=4$ and $X$ is very general, then~$\oev^{d-1}_{\nr,X}$ is not surjective. If~$d\geq 5$, then~$\ev^{d-1}_{\nr,X}$ is surjective.
\item
\label{exaiv}
Let $X$ be an abelian variety of dimension~$d$ over $\Rb$ such that~$X(\Rb)$ has $2^d$ connected components. If $d\in\{1,2,3\}$, or if $d=4$ and $X$ is principally polarized and very general, then~$\oev^{d-1}_{\nr,X}$ is not surjective.
\end{enumerate}
\end{thm}

The variety $X$ in \ref{exai} is an Enriques surface (if $d=2$), or the product of an Enriques surface with an elliptic curve (if $d=3$), as in Theorem \ref{th1} \ref{cohoiii}. In \ref{exaiii} and~\ref{exaiv}, the hypothesis that $X$ is very general means that the assertion holds for all~$X$ as in the statement outside a countable union of nowhere dense algebraic subvarieties of the relevant parameter or moduli space.

Proving that $\ev_{\nr,X}^j$ or $\oev_{\nr,X}^j$ is not surjective is much harder when $j\geq 4$, and we have not been able to devise an efficient criterion to do so, in the spirit of Theorem~\ref{th3bis} for $j=3$ (see Remark \ref{remH4nr}). We therefore leave open the following~questions.

\begin{quests}
\label{qnotsurj}
\begin{enumerate}[label=(\roman*), leftmargin=2em] 
\item 
\label{qi}
Fix $d\geq 4$. Does there exist a smooth projective variety $X$ of dimension $d$ over $\Rb$ such that $\ev^d_{\nr,X}$ is not surjective?
\item
\label{qii}
Fix $d\geq 5$. Does there exist a smooth projective variety $X$ of dimension $d$ over~$\Rb$ such that $\oev^{d-1}_{\nr,X}$ is not surjective?
\end{enumerate}
\end{quests}

In view of the examples of Theorem \ref{th4}, one might expect to answer Question~\ref{qnotsurj}\,\ref{qi} affirmatively by taking $X$ to be the product of an Enriques surface with an abelian variety, and Question \ref{qnotsurj} \ref{qii} with $X$ an abelian variety.

\subsection{\texorpdfstring{$\Ibf$}{I}-cohomology}
\label{parWitt}

Let $X$ be a smooth variety of dimension $d$ over~$\Rb$.

\subsubsection{Signature morphisms}

 The \textit{Witt group} $\Wr(X)$ of $X$ classifies families of quadratic forms parameterized by~$X$. It is defined as the group completion of the monoid of isometry classes of non\-degenerate quadratic forms on vector bundles over~$X$, modulo metabolic forms (\cite{Knebusch1977}, see also~\cite[\S 1.2.1]{Balmer2005}). Let $\Ir(X)\subset\Wr(X)$ be the \textit{fundamental ideal} of forms of even rank. 
Associating with~$\alpha\in \Wr(X)$ and~${x\in X(\Rb)}$ the signature of the real quadratic form~$\alpha|_x$ yields the \textit{signature morphism}
$\sign_X:\Wr(X)\to \Hr^0(X(\Rb),\Zb)$ (see \cite{Mahe1982}).
As a form of even rank has even signature, setting $\sign_{n,X}:=\frac{\sign_X}{2^n}$ for~$n\geq0$~yields
\begin{equation}
\label{signature1}
\sign_{n,X}:\Ir(X)^n \to \Hr^0(X(\Rb),\Zb).
\end{equation}
Let $\Wbf$ and $\Ibf^n$ be the Zariski sheaves on $X$ associated with the presheaves~${U\mapsto \Wr(U)}$ and~${U\mapsto \Ir(U)^n}$. Letting~$\iota:X(\Rb)\to X$ be the natural map, sheafifying \eqref{signature1} induces
\begin{equation}
\label{signature2}
\sign_n:\Ibf^n\to \iota_*\Zb.
\end{equation}

Let $\Lc$ be a line bundle on $X$ and let $\Zb(\Lc)$ be the orientation local system of the real line bundle $\Lc(\Rb)$.
Twisting \eqref{signature2} by $\Lc$ (in the sense of \cite[\S 3.B]{hornbostelRealCycleClass2021}) yields
\begin{equation}
\label{signature5}
\sign_n:\Ibf^n(\Lc)\to \iota_*\Zb(\Lc).
\end{equation}
Taking degree $c$ cohomology in \eqref{signature5} gives rise to \textit{higher signature morphisms}
\begin{equation}
\label{signature3}
\sign^c_{n,X}:\Hr^c(X,\Ibf^n(\Lc))\to \Hr^c(X(\Rb),\Zb(\Lc)).
\end{equation}

Setting~$n=c=0$ and $\Lc=\Oc_X$ in \eqref{signature3}, we get the signature morphism
\begin{equation}
\label{signature4}
\Wr_{\nr}(X)\to\Hr^0(X(\Rb),\Zb)
\end{equation}
on the \textit{unramified Witt group} $\Wr_{\nr}(X):=\Hr^0(X,\Wbf)$ of~$X$.

\subsubsection{Lower bounds on the image of the signature morphisms}

The problem of computing or controlling the image of \eqref{signature4} was first raised by Knebusch, who solved it entirely in the case of curves \cite[Theorem 10.4\,(iv)]{Knebusch1976II}. It turns out that the cokernel~$C_X$ of \eqref{signature4} is a $2$-primary torsion group, and it is natural to focus on estimating the exponent of $C_X$ (see e.g.\,\cite[\S 4]{Parimala1992}). This problem was successively considered by Mah\'e \cite{Mahe1982}, Monnier \cite{Monnier1997,Monnier2000}, Jacobson \cite{jacobsonRealCohomologyPowers2017} and finally Lerbet \cite[Proposition~5.3]{lerbetImageHigherSignature2026}, who proved that $C_X$ is killed by $2^d$. 

We establish an analogous bound for the exponent of the cokernel of higher signature morphisms, conjectured in \cite[Conjecture 6.5]{lerbetImageHigherSignature2026}.

\begin{thm}[Theorem \ref{thexpcoker}]
\label{thexpcokerintro}
Let $X$ be a smooth variety of dimension $d$ over~$\Rb$. Let~$\Lc$ be a line bundle on $X$. For $0\leq c\leq d$ and $0\leq n\leq d$, the cokernel of the higher signature morphism $\sign_{n,X}^c:\Hr^c(X,\Ibf^n(\Lc))\to \Hr^c(X(\Rb),\Zb(\Lc))$ is killed by~$2^{d-n}$.
\end{thm}

Our proof of Theorem \ref{thexpcokerintro} exploits the canonical section of $\oev^d_X$ constructed in \cite[Remark~1.6\,(ii)]{benoistIntegralHodgeConjecture2020} (which strengthens Krasnov and van Hamel's Theorem~\ref{thKvH}).

\subsubsection{Signature morphisms with small image}

We also investigate the optimality of the exponent bounds in Theorem \ref{thexpcokerintro}. 

\begin{thm}[Theorem \ref{thexpcokeroptimal}]
\label{thexpcokeroptimalintro}
Fix $0\leq c\leq d$ and $0\leq n< d$. There exists a smooth variety $X$ of dimension~$d$ over~$\Rb$ such that the cokernel of 
$$\sign_{n,X}^c:\Hr^c(X,\Ibf^n)\to \Hr^c(X(\Rb),\Zb)$$ 
is not killed by~$2^{d-n-1}$. If $d-c\leq 4$, one can moreover choose $X$ to be projective.
 \end{thm}
 
 The question of the optimality of such exponent bounds was first raised by Lerbet in \cite[Conjecture 6.5]{lerbetImageHigherSignature2026}. The precise optimality question asked in \emph{loc.\,cit.} was solved by Hennig in \cite[Proposition~5.1]{hennigRealCycleClass2025}. The first assertion of Theorem \ref{thexpcokeroptimalintro} is a variant of Hennig's result (which treats $n=c$ and allows the signature morphism to be twisted by a line bundle, as in \cite{lerbetImageHigherSignature2026}) and its proof is inspired by his arguments. 
 
  In the spirit of the results presented in \S\ref{parcoho} and \S\ref{parunr}, it is desirable to have projective examples. The second assertion of Theorem \ref{thexpcokeroptimalintro} is what we could prove in this direction. In the crucial case when $n=c=0$, this gives examples of smooth projective varieties~$X$ of dimension $1\leq d\leq 4$ over $\Rb$ such that the cokernel of the signature morphism~${\Wr_{\nr}(X)\to \Hr^0(X(\Rb),\Zb)}$ is not killed by $2^{d-1}$. This result is due to Knebusch \cite[Theorem 10.4\,(iv)]{Knebusch1976II} for $d=1$ and to Monnier \cite[Theorem~5.2]{Monnier2000} for~$d=2$ (note that $\Wr(X)=\Wr_{\nr}(X)$ in these dimensions, see e.g.\,\cite[p.\,128]{balmerWalterGerstenWitt2002}). It is new for~$d\in\{3,4\}$. Our examples are smooth cubic hypersurfaces, and our proof uses Theorem~\ref{th4}\,\ref{exaiii} as a key input. Surprisingly, smooth cubic hypersurfaces cannot yield examples of this kind in higher dimensions.

\begin{thm}[Propositions \ref{prop:cokernel_cubic} and \ref{propPfistercubic}]
Let $X$ be a smooth cubic hypersurface of dimension $d$ over $\Rb$ with $X(\Rb)$ disconnected. 
If $d\in\{1,2,3\}$, or if~$d=4$ and~$X$ is very general, then the cokernel of the signature morphism $\Wr_{\nr}(X)\to \Hr^0(X(\Rb),\Zb)$ is not killed by $2^{d-1}$. If~$d\geq 5$, this cokernel is killed by~$2^{d-1}$.
 \end{thm}

This leaves open the following question.

\begin{quest}
\label{qsignnotsurj}
Fix $d\geq 5$. Does there exist a smooth projective variety $X$ of dimension $d$ over~$\Rb$ with the property that the cokernel of the signature morphism ${\Wr_{\nr}(X)\to \Hr^0(X(\Rb),\Zb)}$ is not killed by $2^{d-1}$?
\end{quest}

\section{Cohomology}

In this section, we study evaluation maps for $G$-equivariant cohomology.

\subsection{Recollections on equivariant cohomology}
\label{parrecollections}

We gather here a few known facts on $G$-equivariant cohomology. Let $X$ be an algebraic variety over $\Rb$.

We first give a precise definition of the evaluation maps under investigation. We work with complexes of $G$-modules, as this generality will be useful in \S\ref{parsurjectiveevnr}. Let~$M$ be a bounded complex of $G$-modules and fix $j\in\Zb$. Recall that $\Hr^j_G(\pt,M)=\Hr^j(G,M)$ for~$j\in\Zb$. There exists a (necessarily unique) evaluation morphism
$$\ev^j_{X,M}:\Hr^j_G(X(\Cb),M)\to \Hr^0(X(\Rb), \Hr^j(G,M))$$
such that $\ev^j_{X,M}(\alpha)$ is given by $x\mapsto\alpha|_x$ for all $\alpha\in\Hr^j_G(X(\Cb),M)$ and $x\in X(\Rb)$. One can construct $\ev^j_{X,M}$ as the composition 
$$\Hr^j_G(X(\Cb),M)\to\Hr^j_G(X(\Rb),M)\to\Hr^0(X(\Rb), \Hr^j(G,M))$$
of the restriction map and of an edge map in the first spectral sequence of equivariant cohomology \cite[Th\'eor\`eme 5.2.1]{tohoku}. We are mostly interested in $\oev^j_X:=\ev^j_{X,\Zb/2}$ (for~$j\geq 0$) and $\ev_X^j:=\ev^j_{X,\Zb(j)}$ (for $j\geq 1$), which were introduced in \eqref{evalcoho}.

For any $G$-equivariant sheaf $\Fc$ on $X(\Cb)$, we consider the $G$-equivariant sheaves $\Fc[G]:=\Fc\otimes_\Zb\Zb[G]$ and $\Fc(1):=\Fc\otimes_\Zb\Zb(1)$, both endowed with the diagonal $G$-action. It follows from \cite[(1.5)]{benoistIntegralHodgeConjecture2020} that
\begin{equation}
\label{Shapiro}
\Hr^j_G(X(\Cb),\Fc[G])=\Hr^j(X(\Cb),\Fc).
\end{equation}
As explained in \cite[(1.6) and \S 1.1.2]{benoistIntegralHodgeConjecture2020}, the short exact sequence \begin{equation}
\label{rcses}
0\to\Fc\to\Fc[G]\to\Fc(1)\to 0
\end{equation}
 then gives rise to a \textit{real-complex long exact sequence}
\begin{equation}
\label{rcles}
 \Hr^j_G(X(\Cb),\Fc)\to \Hr^j(X(\Cb),\Fc)\to \Hr^j_G(X(\Cb),\Fc(1))\to\Hr^{j+1}_G(X(\Cb),\Fc)\to \dots
\end{equation}
with boundary maps given by cup product with the generator of $\Hr^1(G,\Zb(1))\simeq\Zb/2$.

\subsection{Varieties with no \texorpdfstring{$2$}{2}-torsion in cohomological degree \texorpdfstring{$d+1$}{d+1}}

We start with a basic proposition.

\begin{prop}
\label{propno2torsion}
Let $X$ be a smooth variety of dimension $d\geq 1$ over $\Rb$. If the $2$-torsion subgroup~$\Hr^{d+1}(X(\Cb),\Zb)[2]$ of $\Hr^{d+1}(X(\Cb),\Zb)$ vanishes, then $\ev^d_X$ is~surjective.
\end{prop}

\begin{proof}
Consider the short exact sequence of $G$-modules 
\begin{equation}
\label{sesGdd+1}
0\to \Zb[G]\to\Zb(d)\oplus\Zb(d+1)\to\Zb/2\to0.
\end{equation}
Taking $G$-equivariant cohomology in \eqref{sesGdd+1} and using \eqref{Shapiro} yields an exact sequence
\begin{equation}
\label{lesGdd+1}
\Hr^{d}_G(X(\Cb),\Zb(d))\oplus\Hr^{d}_G(X(\Cb),\Zb(d+1))\to\Hr^{d}_G(X(\Cb),\Zb/2)\to \Hr^{d+1}(X(\Cb),\Zb).
\end{equation}
As $\Hr^{d+1}(X(\Cb),\Zb)[2]=0$, the right morphism of \eqref{lesGdd+1} vanishes, and hence the left morphism of \eqref{lesGdd+1} is surjective. Now choose an element $\sigma\in \Hr^0(X(\Rb),\Zb/2)$. By Theorem \ref{thKvH}, there exists a class~$\alpha\in\Hr^d_G(X(\Cb),\Zb/2)$ such that $\oev_{X}^d(\alpha)=\sigma$. Let 
$$(\beta,\gamma)\in \Hr^{d}_G(X(\Cb),\Zb(d))\oplus\Hr^{d}_G(X(\Cb),\Zb(d+1))$$
be a lift of $\alpha$ in \eqref{lesGdd+1}. Since the group $\Hr^d(G,\Zb(d+1))$ vanishes and the reduction mod~$2$ morphism $\Hr^d(G,\Zb(d))\to\Hr^d(G,\Zb/2)$ is an isomorphism, we see that~$\ev_X^d(\beta)=\sigma$. This proves the proposition.
\end{proof}

\begin{cor}
\label{corod=1}
If $X$ is a smooth curve over $\Rb$, then $\ev^1_X$ is surjective.
\end{cor}

\subsection{Hypersurfaces and complete intersections}

The following proposition is a real incarnation of the weak Lefschetz theorem. A different argument based on the Hochschild--Serre spectral sequence appears in \cite[Proof of Proposition 5.1]{BWWu2024}.

\begin{prop}
\label{prop:weak_lefschetz}
Fix $c,d\geq 0$. Let $Y$ be a smooth projective variety over $\Rb$, of dimension $c+d$. Let $\Ec$ be an ample vector bundle of rank $c$ on $Y$. Let $\tau\in \Hr^0(Y,\Ec)$ be such that the zero locus $X\subset Y$ of $\tau$ is smooth of codimension $c$. Let $M$ be a~$G$\nobreakdash-module. Then the restriction homomorphism 
$$\Hr_G^j(Y(\Cb),M)\to\Hr_G^j(X(\Cb),M)$$
is an isomorphism for $j<d$ and is injective if $j=d$.
\end{prop}

\begin{proof}
Using the long exact sequence of $G$-equivariant cohomology of the pair of $G$\nobreakdash-spaces $(Y(\Cb),X(\Cb))$, one reduces to proving the vanishing of $\Hr^j_G(Y(\Cb),X(\Cb),M)$ for $j\leq d$. We prove it, for all $M$, by induction on $j$. Let $f:Y(\Cb)\setminus X(\Cb)\hookrightarrow Y(\Cb)$ be the inclusion. Applying \eqref{rcles} to $\Fc:=f_!M$ yields an exact sequence
\begin{equation}
\label{realcomplex}
\Hr_G^{j-1}(Y(\Cb),X(\Cb),M(1))\to\Hr_G^j(Y(\Cb),X(\Cb),M)\to\Hr^j(Y(\Cb),X(\Cb),M).
\end{equation}
The left term of \eqref{realcomplex} vanishes by the induction hypothesis. By Sommese's weak Lefschetz theorem (see \cite[Theorem~7.1.1]{lazarsfeldPositivityII} whose proof applies
with coefficients in any abelian group $M$), the group $\Hr^j(Y(\Cb),X(\Cb),M)$ vanishes because~$j\leq d$. The middle term of \eqref{realcomplex} therefore vanishes and the proof is complete.
\end{proof}

\begin{prop}
\label{prop:image_global_sign_hypersurface}
Let $X\subset \mathbb{P}_\Rb^{c+d}$ be a smooth complete intersection of codimension~$c$. Fix $0\leq j\leq d-1$. Then the image of $\oev^j_X$ is generated by $1\in \Hr^0(X(\Rb),\Zb/2)$.

In particular, if $X(\Rb)$ is not connected, then $\oev^j_X$ is not surjective.
\end{prop}

\begin{proof}
Set $Y:=\mathbb{P}^{c+d}_{\Rb}$. Consider the commutative diagram
\begin{equation*}
\label{diagoev}
\begin{tikzcd}[row sep=small, column sep=normal]
	{\Hr_G^{j}(Y(\Cb),\Zb/2)}&{\Hr^0(Y(\Rb),\Zb/2)} \\
	{\Hr_G^{j}(X(\Cb),\Zb/2)} &{\Hr^0(X(\Rb),\Zb/2)}
	\arrow["{\oev^j_{Y}}", from=1-1, to=1-2]
	\arrow["\rotatebox{90}{$\sim$}"', from=1-1, to=2-1]
	\arrow[from=1-2, to=2-2]
	\arrow["{\oev^j_{X}}", from=2-1, to=2-2]
\end{tikzcd}
\end{equation*}
whose vertical arrows are restriction maps. The left vertical arrow is bijective for~$j\leq d-1$ by Proposition \ref{prop:weak_lefschetz}. The image of~$\oev^j_X$ is therefore included in the image of $\Hr^0(Y(\Rb),\Zb/2)\to \Hr^0(X(\Rb),\Zb/2)$. As $Y(\Rb)$ is connected, this shows that the image of $\oev^j_X$ is included in the subgroup generated by $1$ (hence equal to it as one sees by considering the classes in the image of $\Hr^j_G(\pt,\Zb/2)\to \Hr^j_G(X(\Cb),\Zb/2)$).
\end{proof}

Here are interesting examples of varieties to which Proposition \ref{prop:image_global_sign_hypersurface} applies.

\begin{cor}
\label{corocubic}
Let $X$ be a projective variety of dimension $d\geq 1$ over $\Rb$ that is either a smooth cubic hypersurface, or a smooth complete intersection of two quadrics. Suppose that $X(\Rb)$ is not connected. Then $\oev_X^{d-1}$ is not surjective.
\end{cor}

\subsection{Restriction to the real locus}
\label{parrestr}

Let $X$ be a smooth proper variety of dimension~$d$ over $\Rb$. For $j\in\Zb$, consider the composition
\begin{equation}
\label{restrdecompo}
r_X^j:\Hr^j_G(X(\Cb),\Zb/2)\to\Hr_G^j(X(\Rb),\Zb/2)\isoto\bigoplus_{i=0}^j\Hr^i(X(\Rb),\Zb/2)
\end{equation}
of the restriction to the real locus and of the canonical decomposition \cite[(1.26)]{benoistIntegralHodgeConjecture2020}. Let $V^j_X\subset \bigoplus_{i=0}^j\Hr^i(X(\Rb),\Zb/2)$ be the image of $r^j_X$. For $\alpha\in \Hr^j_G(X(\Cb),\Zb/2)$, write 
$$r^j_X(\alpha)=([\alpha]_0,\dots,[\alpha]_j)\quad\text{in }\bigoplus_{i=0}^j\Hr^i(X(\Rb),\Zb/2).$$

 To prove Theorem \ref{thm:topological_argument_integral_case} in \S\ref{parprodEnriques}, we need to control $V_X^j$ for $j=d$. Theorem~\ref{thm:image_of_total_restriction_to_the_real_locus} below allows us to do so, as an application of the duality theorem \cite[Theorem~1.12]{benoistIntegralHodgeConjecture2020}. To state it, we introduce some notation. Let $w_X=(w_{i,X})_{i\geq 0}\in\Hr^*(X(\Rb),\Zb/2)$ be the total Stiefel--Whitney class of the tangent bundle of $X(\Rb)$. Let
$$\deg:\Hr^*(X(\Rb),\Zb/2)\to \Zb/2$$
be the composition of the projection $\Hr^*(X(\Rb),\Zb/2)\to \Hr^d(X(\Rb),\Zb/2)$ and of the sum $\Hr^d(X(\Rb),\Zb/2)\to\Zb/2$ of the degree maps of all connected components of~$X(\Rb)$.
Poincar\'e duality and the invertibility of $w_X\in \Hr^*(X(\Rb),\Zb/2)$ show that the pairing 
\begin{equation}
\label{Poincare}
(x,y)\mapsto\deg(x\smile y\smile w_X^{-1})
\end{equation} 
on $\Hr^*(X(\Rb),\Zb/2)$ is nondegenerate. 

\begin{thm}
\label{thm:image_of_total_restriction_to_the_real_locus}
Let $X$ be a smooth proper variety of dimension $d$ over~$\Rb$. For~$j\in\Zb$, the subspaces $V^j_X$ and $V_X^{2d-j-1}$ of $\Hr^*(X(\Rb),\Zb/2)$  are exact orthogonal complements with respect to the nondegenerate pairing \eqref{Poincare}.
\end{thm}

\begin{proof}
We assume that $j\geq d$ (if not, replace $j$ with $2d-j-1$). Consider the~diagram
\vspace{-10pt}
\begin{equation*}
\begin{tikzcd}[row sep=small, column sep=small]
	0 & {\Hr_{G,X(\Rb)}^j(X(\Cb),\Zb/2)} & {\Hr_G^j(X(\Cb),\Zb/2)} & {\Hr_G^j(X(\Cb)\setminus X(\Rb),\Zb/2)} & 0\\
	& \bigoplus_{i=0}^{j-d}\Hr^i(X(\Rb),\Zb/2)& \Hr^*(X(\Rb),\Zb/2)& &
	\arrow[from=1-1, to=1-2]
	\arrow[from=1-2, to=1-3]
	\arrow[from=1-3, to=1-4]
        \arrow["\tau"'{xshift=7pt},shift right, curve={height=5pt}, yshift=0pt, start anchor={[xshift=12pt]}, end anchor={[xshift=-16pt]}, from=1-4, to=1-3]
        	\arrow[from=1-4, to=1-5]
	\arrow[equals, from=1-2, to=2-2]
	\arrow["r_X^j", from=1-3, to=2-3]
	\arrow["\smile w_X"{yshift=1.5pt}, from=2-2, to=2-3]
\end{tikzcd}
\end{equation*}
whose top row is the exact sequence of \cite[Proposition 1.3]{benoistIntegralHodgeConjecture2020}, where~$\tau$ is the section induced by the canonical splitting described in \cite[Proof of Proposition~1.3]{benoistIntegralHodgeConjecture2020} and whose left vertical equality stems from equivariant purity \cite[(1.20)]{benoistIntegralHodgeConjecture2020} and the canonical decomposition \cite[(1.26)]{benoistIntegralHodgeConjecture2020}. 

The square commutes by \cite[(1.39) and Remark 1.6\,(i)]{benoistIntegralHodgeConjecture2020}. It follows~that
\begin{equation}
\label{rj1}
r_X^j(\Hr_{G,X(\Rb)}^j(X(\Cb),\Zb/2))=\Big( \bigoplus_{i=0}^{j-d}\Hr^i(X(\Rb),\Zb/2)\Big)\smile w_X.
\end{equation}
For $k\in\Zb$, let $z_k:\Hr^k_G(X(\Cb)\setminus X(\Rb),\Zb/2)\to \bigoplus_{i=k-d+1}^{k}\Hr^i(X(\Rb),\Zb/2)$ be the map called $w_k$ in \cite[p.\,25]{benoistIntegralHodgeConjecture2020} (we change notation to avoid confusion with the Stiefel--Whitney classes $w_{i,X}$ of $X(\Rb)$). The very definitions of the maps $\tau$ and $z_j$ show that $z_j=(\smile w_X^{-1})\circ r_X^j\circ\tau$ 
and hence that 
\begin{equation}
\label{rj2}
\Big(r^j_X\circ\tau(\Hr_G^j(X(\Cb)\setminus X(\Rb),\Zb/2))\Big)\smile w_X^{-1}=\Im(z_j).
\end{equation}
Combining \eqref{rj1} and \eqref{rj2}, we see that 
\begin{equation}
\label{Rjw-1}
V^j_X\smile w_X^{-1}=\Big(\bigoplus_{i=0}^{j-d}\Hr^i(X(\Rb),\Zb/2)\Big)\oplus\Im(z_j)
\end{equation}
as subgroups of $\Hr^*(X(\Rb),\Zb/2)$. 
By \eqref{Rjw-1}, the orthogonal complement of $V^j_X$ with respect to \eqref{Poincare} is exactly the subgroup~$W\subset \bigoplus_{i=0}^{2d-j-1}\Hr^i(X(\Rb),\Zb/2)$ that is the orthogonal complement of $\Im(z_j)$ with respect to the nondegenerate pairing 
$$\Big(\bigoplus_{i=0}^{2d-j-1}\Hr^i(X(\Rb),\Zb/2)\Big)\times \Big(\bigoplus_{i=j-d+1}^{d}\Hr^i(X(\Rb),\Zb/2)\Big)\to \Zb/2$$
given by $(x,y)\mapsto \deg(x\smile y)$.

The duality theorem \cite[Theorem 1.12]{benoistIntegralHodgeConjecture2020} implies that the images of $z_j$ and~$z_{2d-j-1}$ are orthogonal complements of each other with respect to the nondegenerate pairing $(x,y)\mapsto \deg(x\smile y\smile w_X)$. We deduce from this fact that $W$ is the image of~$\im(z_{2d-j-1})\smile w_X$ by the projection to $\bigoplus_{i=0}^{2d-j-1}\Hr^i(X(\Rb),\Zb/2)$. The very definition of $z_{2d-j-1}$ now implies that ${W=V_X^{2d-j-1}}$, as wanted.
\end{proof}

\begin{rems}
\begin{enumerate}[label=(\roman*), leftmargin=2em] 
\item
\label{qui}
When $X$ is Galois maximal, Theorem \ref{thm:image_of_total_restriction_to_the_real_locus} is proved by Kharlamov and R\u{a}sdeaconu in \cite[Propositions 2.6\,(2) and 2.17]{KharlamovRasdeaconu2026}.
\item
\label{quii}
Theorem \ref{thm:image_of_total_restriction_to_the_real_locus} is reminiscent of \cite[Theorem 3.3]{benoistIntegralHodgeConjecture2020} (and both are applications of the duality theorem \cite[Theorem 1.12]{benoistIntegralHodgeConjecture2020}).
\item 
\label{quiii}
It is a matter of pure linear algebra to deduce \cite[Proposition~8.1]{benoistPeriodindexProblemReal2019} from Theorem \ref{thm:image_of_total_restriction_to_the_real_locus}. We think of Theorem \ref{thm:image_of_total_restriction_to_the_real_locus} as a generalization of this proposition, with a more elegant statement.
\item
\label{quiii}
As~$V_X^{d-1}\subset V_X^d$ are exact orthogonal complements with respect to a nondegenerate pairing on~$\Hr^*(X(\Rb),\Zb/2)$, it follows from Theorem \ref{thm:image_of_total_restriction_to_the_real_locus} that
$$\dim_{\Zb/2}(V_X^{d-1})\leq\frac{1}{2}\dim_{\Zb/2}(\Hr^*(X(\Rb),\Zb/2))\leq\dim_{\Zb/2}(V_X^{d}).$$
\end{enumerate}
\end{rems}

\subsection{Products with Enriques surfaces}
\label{parprodEnriques}

Let $S$ be an Enriques surface over $\Rb$. Then there exist exactly two finite \'etale double covers $S'_i\to S$ (for $i\in\{1,2\}$) such that $S'_i$ is geometrically connected (and hence a K3 surface). The images in~$S(\Rb)$ of the $S'_i(\Rb)$ are called the \emph{halves} of $S$ and form a partition of $S(\Rb)$ (see \cite[\S1.3]{degtyarevHalvesRealEnriques1996}).

\begin{thm}\label{thm:topological_argument_integral_case}
Let $E$ be an elliptic curve over $\Rb$ such that $E(\Rb)$ is not connected. Let $S$ be an Enriques surface over $\Rb$. Fix $d\geq 2$ and set~$X:=S\times E^{d-2}$. Then $\ev^d_X$ is surjective if and only if at least one of the following conditions holds:
\begin{enumerate}[label=(\roman*)] 
\item all connected components of $S(\Rb)$ have even Euler characteristic;
\item both halves of $S$ are nonempty.
\end{enumerate}
\end{thm}

\begin{proof}
We proceed in two steps. We freely use the notation introduced in \S\ref{parrestr}.

\begin{Step}
\label{Step1}
Assuming that $S(\Rb)$ has a connected component $\Pi_S$ of odd Euler characteristic and that only one half of $S$ is nonempty, we prove that $\ev^d_X$ is not surjective.
\end{Step}

Let $S'\to S$ be the degree $2$ finite \'etale cover of $S$ by a K3 surface that is surjective at the level of real points. Let $\alpha\in \Hr^1_G(S(\Cb),\Zb/2)=\Hr^1_{\et}(S,\Zb/2)$ (see \cite{coxEtaleHomotopyType1979} for this isomorphism) be the class of this cover. One has~$[\alpha]_0=0$ in~$\Hr^0(S(\Rb),\Zb/2)$ because the covering $S'\to S$ is trivial over the points of $S(\Rb)$. In addition, the class~$[\alpha]_1$ is the Borel--Haefliger class of the $2$\nobreakdash-torsion line bundle associated with~$S'\to S$ (see e.g.\,\cite[Lemma 2.1]{benoistPeriodindexProblemReal2019}). As this $2$\nobreakdash-torsion line bundle is nontrivial (because~$S'$ is geometrically connected), it must necessarily be the canonical bundle of~$S$ (because~$S$ is an Enriques surface), and we deduce from \cite[Theorem~3.2.1\,1)]{krasnovCharacteristicClasses1991} that~${[\alpha]_1=w_{1,S}}$ in~$\Hr^1(S(\Rb),\Zb/2)$. We have shown that $r^1_S(\alpha)=(0,w_{1,S})$.

Applying Theorem \ref{thm:image_of_total_restriction_to_the_real_locus} to $E$ (noting that $w_E=1$ because the tangent bundle is trivial) shows the existence of $\beta\in \Hr^1_G(E(\Cb),\Zb/2)$ with $r_E^1(\beta)=(0,[E(\Rb)])$, where~$[E(\Rb)]\in \Hr^1(E(\Rb),\Zb/2)$ is the fundamental class of $E(\Rb)$. (One could also construct $\beta$ directly as the class of an explicit double cover of $E$.)

Let $p:X\to S$ and $(p_i:X\to E)_{1\leq i\leq d-2}$ be the projection maps. Define
$$\gamma:=p^*\alpha\smile p_1^*\beta\smile\dots\smile p_{d-2}^*\beta\quad\text{in }\Hr^{d-1}_G(X(\Cb),\Zb/2).$$ Let~${q:X(\Rb)\to S(\Rb)}$ and $(q_i:X(\Rb)\to E(\Rb))_{1\leq i\leq d-2}$ be the projection maps. Set
\begin{equation}
\label{formuladelta}
\delta:=q^*w_{1,S}\smile q_1^*[E(\Rb)]\smile\dots\smile q_{d-2}^*[E(\Rb)]\quad\text{in }\Hr^{d-1}(X(\Rb),\Zb/2).
\end{equation}
As $r^1_S(\alpha)=(0,w_{1,S})$ and $r_E^1(\beta)=(0,[E(\Rb)])$, the formula \cite[(1.28)]{benoistIntegralHodgeConjecture2020} shows that $r^{d-1}_X(\gamma)=(0,\dots,0,\delta)$ in $\bigoplus_{i=0}^{d-1}\Hr^i(X(\Rb),\Zb/2)$.

Let $\Pi_E$ be any connected component of $E(\Rb)$, and define $\Pi:= \Pi_S\times(\Pi_E)^{d-2}$. Let~$\sigma \in \Hr^0(X(\Rb),\Zb/2)$ be the element taking the value $1$ on $\Pi$ and $0$ elsewhere. We claim that $\sigma$ is not in the image of $\ev^d_X$. Assume for contradiction that~${\ev^d_X(\varepsilon)=\sigma}$ for some $\varepsilon\in \Hr^d_G(X(\Cb),\Zb(d))$. Let $\ovarepsilon\in\Hr^d_G(X(\Cb),\Zb/2)$ be the reduction mod~$2$ of~$\varepsilon$. It follows from~\mbox{\cite[(1.33)]{benoistIntegralHodgeConjecture2020}} that $[\ovarepsilon]_1=0$. 

The equalities~$([\ovarepsilon]_0,[\ovarepsilon]_1)=(\sigma,0)$ and $r^{d-1}_X(\gamma)=(0,\dots,0,\delta)$,
combined with the fact that~$w_X=q^*w_S$ (because the tangent bundle of $E$ is trivial), show that 
\begin{equation}
\label{triple2}
r_X^d(\ovarepsilon)\smile r^{d-1}_X(\gamma)\smile w_X^{-1}=(0,\dots,0,\sigma\delta,q^*w_{1,S}\smile\sigma\delta)
\end{equation}
in $\Hr^*(X(\Rb),\Zb/2)$. Theorem \ref{thm:image_of_total_restriction_to_the_real_locus} implies that 
\begin{equation}
\label{triple1}
\deg(r_X^d(\ovarepsilon)\smile r^{d-1}_X(\gamma)\smile w_X^{-1})=0.
\end{equation}
In view of \eqref{formuladelta}, the degree of the class \eqref{triple2} is equal to $\deg_{\Pi_S}(w_1(\Pi_S)^2)$, that is, to $\chi(\Pi_S)$ mod $2$. Since this degree vanishes by \eqref{triple1}, we deduce that $\Pi_S$ has even Euler characteristic. This is the required contradiction. 

\begin{Step}
Assuming that all connected components of $S(\Rb)$ have even Euler characteristic or that both halves of $S$ are nonempty, we prove that $\ev^d_X$ is surjective.
\end{Step}

As $\ev^1_E$ is surjective by Corollary \ref{corod=1}, it suffices to show that $\ev^2_S$ is surjective. 

We may assume that $S(\Rb)\neq\varnothing$. Set $A:=\{\alpha\in \Hr^1_G(S(\Cb),\Zb/2)\mid [\alpha]_0=0\}$. The group $\Hr^1_G(S(\Cb),\Zb/2)$ has four elements, corresponding to the trivial cover~$S\sqcup S\to S$, to the cover $S_{\Cb}\to S$, and to the two K3 covers of $S$. If both halves of $S$ are nonempty, then $A$ contains only the trivial element. Otherwise, it has two elements: the trivial class and the class $\alpha$ of the K3 cover~$S'\to S$ which is surjective at the level of real points. The latter satisfies $[\alpha]_1=w_{1,S}$ (see Step \ref{Step1}), so~$[\alpha]_1\smile w_{1,S}=0$ if all connected components of $S(\Rb)$ have even Euler characteristic.

Define $B:=\{\beta\in \Hr^2_G(S(\Cb),\Zb/2)\mid [\beta]_1=0\}$. It follows from \cite[Proposition~8.1]{benoistPeriodindexProblemReal2019} 
that $(\oev^2_S)|_B$ is surjective if and only if $[\alpha]_1\smile w_{1,S}=0$ in $\Hr^2(S(\Rb),\Zb/2)$ for all $\alpha\in A$. By the above, this property is satisfied under our hypothesis on $S$.

Now fix $\sigma\in \Hr^0(S(\Rb),\Zb/2)$, and choose $\beta\in B$ such that $\oev^2_S(\beta)=\sigma$. Consider the exact commutative diagram of $G$-modules
\begin{equation}
\label{commexact}
\begin{tikzcd}[row sep=small, column sep=small]
	{0}&{\Zb(3)} & {\Zb[G]} & {\Zb(2)} & {0} \\
	{0}&{\Zb/2} & {\Zb/2[G]} & {\Zb/2} & {0}
	\arrow[from=1-1, to=1-2]
	\arrow[from=1-2, to=1-3]
	\arrow[from=1-2, to=2-2]
	\arrow[from=1-3, to=1-4]
	\arrow[from=1-3, to=2-3]
	\arrow[from=2-1, to=2-2]
	\arrow[from=2-2, to=2-3]
	\arrow[from=2-3, to=2-4]
	\arrow[from=1-4, to=2-4]
	\arrow[from=1-4, to=1-5]
	\arrow[from=2-4, to=2-5]
\end{tikzcd}
\end{equation}
whose vertical arrows are reduction mod $2$ morphisms. Taking $G$-equivariant cohomology in \eqref{commexact} and using \eqref{Shapiro} yields a commutative exact diagram
\begin{equation}
\label{commexact2}
\begin{tikzcd}[row sep=small, column sep=small]
	{\Hr^2_G(S(\Cb),\Zb(2))} & {\Hr^3_G(S(\Cb),\Zb(3))} & {\Hr^3(S(\Cb),\Zb)} \\
	{\Hr^2_G(S(\Cb),\Zb/2)} & {\Hr^3_G(S(\Cb),\Zb/2)} & {\Hr^3(S(\Cb),\Zb/2)}
	\arrow[from=1-1, to=2-1]
	\arrow[from=1-1, to=1-2]
	\arrow[from=1-2, to=1-3]
	\arrow[from=1-2, to=2-2]
	\arrow["\rotatebox{90}{$\sim$}"', from=1-3, to=2-3]
	\arrow[from=2-1, to=2-2]
	\arrow[from=2-2, to=2-3]
\end{tikzcd}
\end{equation}
whose right vertical arrow is an isomorphism because $S$ is an Enriques surface.

Let $\gamma\in \Hr^3_G(S(\Cb),\Zb/2)$ be the image of $\beta$ in \eqref{commexact2}. We claim that $\gamma$ lifts to a class $\delta\in\Hr^3_G(S(\Cb),\Zb(3))$ along the middle vertical arrow of \eqref{commexact2}. Taking this claim for granted, diagram \eqref{commexact2} shows that~$\delta$ further lifts to a class $\varepsilon\in \Hr^2_G(S(\Cb),\Zb(2))$. As 
$$\ev^2_S(\varepsilon)=\ev^3_S(\delta)=\oev^3_S(\gamma)=\oev^2_S(\beta)=\sigma,$$
this finishes the proof of the theorem.

We now prove the claim. Consider the commutative diagram
\begin{equation}
\label{compositions}
\begin{tikzcd}[row sep=small, column sep=tiny]
	{\Hr^3_G(S(\Cb),\Zb/2)} & {\Hr^4_G(S(\Cb),\Zb(3))} & {\Hr^4_G(S(\Cb),\Zb/2)} & {\Hr^5_G(S(\Cb),\Zb/2)} \\
	{\Hr^3_G(S(\Rb),\Zb/2)} & {\Hr^4_G(S(\Rb),\Zb(3))} & {\Hr^4_G(S(\Rb),\Zb/2)} & {\Hr^5_G(S(\Rb),\Zb/2)}\\
	{\Hr^*(S(\Rb),\Zb/2)} & {\Hr^1(S(\Rb),\Zb/2)} & {\Hr^*(S(\Rb),\Zb/2)} & {\Hr^*(S(\Rb),\Zb/2)}
	\arrow[from=1-1, to=2-1]
	\arrow[from=1-1, to=1-2]
	\arrow[hook, from=1-2, to=1-3]
	\arrow["\sim"', from=1-3, to=1-4]
	\arrow[from=1-2, to=2-2]
	\arrow[from=1-3, to=2-3]
	\arrow[from=2-1, to=2-2]
	\arrow[from=2-2, to=2-3]
	\arrow[from=2-3, to=2-4]
	\arrow["\rotatebox{90}{$\sim$}"', from=1-4, to=2-4]
	\arrow[equal, from=2-1, to=3-1]
	\arrow[equal, from=2-2, to=3-2]
	\arrow[equal, from=2-3, to=3-3]
	\arrow[equal, from=2-4, to=3-4]
\end{tikzcd}
\end{equation}
whose vertical maps are restriction maps, whose vertical equalities are the canonical decompositions \cite[(1.26) and (1.31)]{benoistIntegralHodgeConjecture2020}, and whose horizontal maps are, from left to right, boundary maps associated with $0\to\Zb(3)\xrightarrow{2}\Zb(3)\to\Zb/2\to 0$, reduction mod $2$ morphisms, and boundary maps of ${0\to\Zb/2\to\Zb/2[G]\to\Zb/2\to 0}$. The right vertical arrow of \eqref{compositions} is bijective by \cite[\S 1.1.3]{benoistIntegralHodgeConjecture2020}. The top right horizontal arrow of \eqref{compositions} is bijective by \eqref{rcles} applied with $\Fc=\Zb/2$ since $\Hr^5(S(\Cb),\Zb/2)=0$ (for dimension reasons) and $\Hr^4_G(S(\Cb),\Zb/2)\to \Hr^4(S(\Cb),\Zb/2)=\Zb/2$ is onto (it sends the class of a real point to a generator). The top middle horizontal arrow of~\eqref{compositions} is injective because $\Hr^4_G(S(\Cb),\Zb(3))$ is $2$-torsion. To see it, use \eqref{rcles} with~$\Fc=\Zb(2)$ and note that $\Hr^4_G(S(\Cb),\Zb(2))\to \Hr^4(S(\Cb),\Zb)=\Zb$ is onto (it sends the class of a real point to a generator) and that $\Hr^5_G(S(\Cb),\Zb(2))\isoto \Hr^5_G(S(\Rb),\Zb(2))$ is $2$-torsion (combine \cite[\S 1.1.3 and (1.31)]{benoistIntegralHodgeConjecture2020}).

As $[\beta]_1=0$, it follows from \cite[\S 1.2.4]{benoistIntegralHodgeConjecture2020} that $[\gamma]_1=0$. The composition $\Hr^3_G(S(\Rb),\Zb/2)\to \Hr^4_G(S(\Rb),\Zb/2)$ of the bottom left horizontal arrows of \eqref{compositions} is the boundary map associated with~$0\to\Zb/2\xrightarrow{}\Zb/4(3)\to\Zb/2\to 0$, and hence is the sum of the Bockstein and of the boundary map of ${0\to\Zb/2\to\Zb/2[G]\to\Zb/2\to 0}$ (see \cite[\S 2.2]{benoistPeriodindexProblemReal2019}). 
It now follows from \cite[(2.8)]{benoistPeriodindexProblemReal2019} and \cite[\S 1.2.4]{benoistIntegralHodgeConjecture2020} that the composition of the three bottom horizontal arrows of \eqref{compositions} is given, at the level of the canonical decompositions, by the formula $(a_0,a_1,a_2)\mapsto (0,a_1,a_1^2)$ (where~$a_i\in \Hr^i(S(\Rb),\Zb/2)$). Since $[\gamma]_1=0$, we deduce that $\gamma$ is mapped to zero in the bottom right group of~\eqref{compositions}. Diagram \eqref{compositions} therefore shows that $\gamma$ is mapped to zero in $\Hr^4_G(S(\Cb),\Zb(3))$ and hence that it lifts to a class $\delta\in\Hr^3_G(S(\Cb),\Zb(3))$.
\end{proof}

\begin{cor}
\label{corevd}
For all $d\geq 2$, there exists a smooth projective variety $X$ of dimension~$d$ over $\Rb$ such that $\ev^d_X$ is not surjective.
\end{cor}

\begin{proof}
This follows from Theorem \ref{thm:topological_argument_integral_case} as there exist real Enriques surfaces with the required property (see \cite[Appendix C, pp.\,228--230]{degtyarevRealEnriquesSurfaces2000}).
\end{proof}

We record the case of Enriques surfaces, a result in the spirit of the work of Mangolte and van Hamel on the topology of real Enriques surfaces \cite{MangolteVanHamel1998}.

\begin{cor}
\label{corEnriques}
Let $S$ be an Enriques surface over $\Rb$. Then $\ev^2_S$ is surjective if and only if at least one of the following conditions holds:
\begin{enumerate}[label=(\roman*)] 
\item all connected components of $S(\Rb)$ have even Euler characteristic;
\item both halves of $S$ are nonempty.
\end{enumerate}
\end{cor}

\subsection{Abelian varieties}
\label{partopova}

The following result is used in the proof of Proposition~\ref{propva2}.

\begin{prop}
\label{propva}
Let $X$ be an abelian variety of dimension $d\geq 1$ over $\Rb$ such that~$X(\Rb)$ has $2^d$ connected components. Then $\oev^{d-1}_X$ is not surjective.
\end{prop}

\begin{proof}
The $G$-module $\Lambda:=\Hr_1(X(\Cb),\Zb)$ is isomorphic to $(\Zb\oplus\Zb(1))^{\oplus a}\oplus \Zb[G]^{\oplus (d-a)}$ for some $a\geq 0$, and $X(\Cb)$ is $G$-equivariantly diffeomorphic to $(\Lambda\otimes_{\Zb}\Rb)/\Lambda$ 
(see e.g.\,\cite[\S\S4--5]{silholRealAbelianVarieties1982}). It follows that  the $G$-equivariant diffeomorphism class of $X(\Cb)$ only depends on the number of connected components of $X(\Rb)$, so we may assume that $X=E^d$ for some elliptic curve $E$ over $\Rb$ such that~$E(\Rb)$ has $2$ connected components. Let $(p_i:X\to E)_{1\leq i\leq d}$ be the projections.

Applying Theorem \ref{thm:image_of_total_restriction_to_the_real_locus} to $E$ (noting that $w_E=1$ because the tangent bundle of~$E$ is trivial) shows the existence of $\alpha\in \Hr^1_G(E(\Cb),\Zb/2)$ with $r_E^1(\alpha)=(0,[E(\Rb)])$, where~$[E(\Rb)]\in \Hr^1(E(\Rb),\Zb/2)$ is the fundamental class of $E(\Rb)$. Define
$$\beta:=p_1^*\alpha\smile\dots\smile p_d^*\alpha\quad\text{in }\Hr^d_G(X(\Cb),\Zb/2).$$
It follows from \cite[(1.28)]{benoistIntegralHodgeConjecture2020} that $r^d_X(\beta)=(0,\dots, 0,[X(\Rb)])$. We deduce from Theorem \ref{thm:image_of_total_restriction_to_the_real_locus} applied to $X$ (noting that $w_X=1$ because the tangent bundle of~$X$ is trivial) that for all $\gamma\in \Hr^{d-1}_G(X(\Cb),\Zb/2)$, one has ${\deg(r_X^{d-1}(\gamma)\smile r_X^d(\beta))=0}$. This exactly means that the number of connected components of $X(\Rb)$ on which~$\oev_X^{d-1}(\gamma)$ is nonzero is even, which proves the proposition.
\end{proof}

\section{Unramified cohomology}

This section is devoted to the study of evaluation maps for unramified cohomology.

Let $X$ be a smooth algebraic variety over $\Rb$. Let~$M$ be a bounded complex of $G$\nobreakdash-modules. Fix~$j\in\Zb$. We define $\ev^j_{\nr,X,M}:\Hr^j_{\nr}(X,M)\to \Hr^0(X(\Rb),\Hr^j(G,M))$ to be induced by $\ev^j_{X,M}$. The particular cases $\oev^j_{\nr,X}:=\ev^j_{\nr,X,\Zb/2}$ (for~$j\geq 0$) and~$\ev^j_{\nr,X}:=\ev^j_{\nr,X,\Zb(j)}$ (for~$j\geq 1$) have already been introduced in~\eqref{unreval}.

\subsection{Surjective evaluation maps}
\label{parsurjectiveevnr}

Recall that a cohomology class is said to have \emph{coniveau} $\geq c$ if it vanishes when restricted to the complement of some closed algebraic subset of codimension~$\geq c$ and that a cohomology group is said to have \emph{coniveau}~$\geq c$ if all its elements have coniveau~$\geq c$.

\subsubsection{A mechanism to construct unramified cohomology classes}

\begin{lem}
\label{thtriangle}
Let $X$ be a smooth variety over $\Rb$. Let $M\xrightarrow{a} N\xrightarrow{b} P\xrightarrow{c}$ be a distinguished triangle in the bounded derived category of $G$-modules and fix $j\in\Zb$.  
The image of the composition
$$\Hr^{j-1}_{\nr}(X,P)\xrightarrow{\ev^{j-1}_{\nr,X,P}}\Hr^0(X(\Rb),\Hr^{j-1}(G,P))\xrightarrow{c_*} \Hr^0(X(\Rb),\Hr^j(G,M))$$
contains the image by $\ev^j_{X,M}$ of $\{\alpha\in \Hr^j_G(X(\Cb),M)\mid a_*\alpha\textrm{ has coniveau }\geq 2\}$.
\end{lem}

\begin{proof}
Let $\alpha\in\Hr^j_G(X(\Cb),M)$ be a class such that $a_*\alpha$ has coniveau $\geq 2$. Let~${Z\subset X}$ be a closed subset of codimension $\geq 2$, with complement $U:=X\setminus Z$, such that~$(a_*\alpha)|_{U}=0$ in~$\Hr^j_G(U(\Cb),N)$. The long exact sequence associated with the triangle yields the existence of $\beta\in \Hr^{j-1}_G(U(\Cb),P)$ such that $c_*\beta=\alpha|_{U}$.

Bloch--Ogus theory identifies the unramified cohomology group $\Hr^{j-1}_{\nr}(X,P)$ with the subgroup of $\colim_V \Hr^{j-1}_G(V(\Cb),P)$ (where $V$ runs over all dense open subsets of $X$) consisting of those classes whose residues along all codimension $1$ integral subvarieties of $X$ vanish (see \cite[(5.1) and (5.2)]{benoistIntegralHodgeConjecture2020}, and note that the arguments given there apply to bounded complexes of $G$-modules as well as to $G$-modules). 

As $Z$ has codimension $\geq 2$, it follows that $\beta$ induces an unramified cohomology class $\gamma\in \Hr^{j-1}_{\nr}(X,P)$. By construction, the class $c_*\gamma\in \Hr^j_{\nr}(X,M)$ is the unramified cohomology class associated with $\alpha$. It follows that $c_*\ev^{j-1}_{\nr,X,P}(\gamma)=\ev^j_{X,M}(\alpha)$.
\end{proof}

\subsubsection{A few distinguished triangles}

We now construct interesting complexes of $G$-modules that will be used to apply Lemma \ref{thtriangle}.
Fix $d\in\Zb$. Consider a
diagram 
\begin{equation}
\label{diagtriangle}
\begin{tikzcd}[row sep=small, column sep=small]
	{\Zb[G]} & {\Zb\oplus\Zb(1)} & {\Zb/2} & {} \\
	{\Zb[G]} & {\Zb[G]} & {K_d} & {}\\
	{} & {L_d} & {L_d} & \\
	{} & {} & {}
	\arrow[from=1-1, to=1-2]
	\arrow[equals, from=1-1, to=2-1]
	\arrow[from=1-2, to=1-3]
	\arrow["\varphi"', from=1-2, to=2-2]
	\arrow[from=1-3, to=1-4]
	\arrow[from=1-3, to=2-3]
	\arrow["\psi"', from=2-1, to=2-2]
	\arrow[from=2-2, to=2-3]
	\arrow[from=2-3, to=2-4]
	\arrow[from=2-2, to=3-2]
	\arrow[from=2-3, to=3-3]
	\arrow[equals, from=3-2, to=3-3]
	\arrow[from=3-2, to=4-2]
	\arrow[from=3-3, to=4-3]
\end{tikzcd}
\end{equation}
of distinguished triangles in the bounded derived category of $G$-modules, whose upper triangle is induced by the short exact sequence $0\to\Zb[G]\to\Zb\oplus\Zb(1)\to\Zb/2\to 0$, where the morphisms $\varphi$ and $\psi$ have saturated images isomorphic to $\Zb(d)$, the complexes~$K_d$ and $L_d$ are defined by the fact that the second row and the second column are distinguished, the morphism $\Zb/2\to K_d$ exists by axiom (TR3) of triangulated categories, and the rest of the diagram can be completed thanks to the octahedral axiom of triangulated categories.

\begin{lem}
\label{lemLd}
For all $d\in\Zb$, one has $L_d\simeq \Zb(d-1)\oplus\Zb(d-1)[1]$.
\end{lem}

\begin{proof}
The second column of \eqref{diagtriangle} shows that the nonzero cohomology $G$-modules of~$L_d$ are $\Zb(d-1)$ in degrees $-1$ and $0$, so one has a distinguished triangle 
\begin{equation}
\label{sesL}
\Zb(d-1)[1]\to L_d\to\Zb(d-1)\to.
\end{equation}
The class of this extension is given by a morphism $\Zb(d-1)\to\Zb(d-1)[2]$, which corresponds to a cohomology class $\alpha\in\Hr^2(G,\Zb)=\Zb/2$. The claim is then that this class is trivial. Suppose on the contrary that $\alpha$ is the nontrivial element of $\Hr^2(G,\Zb)$. As the connecting homomorphisms \[\Hr^j(G,\Zb(d-1))\to\Hr^{j+2}(G,\Zb(d-1))\] of the long exact sequence induced by \eqref{sesL} are given by the cup product with $\alpha$, they are isomorphisms for all $j\geq 1$. It follows that~$\Hr^j(G,L_d)$ vanishes for $j\geq 2$. 

However, the second row of \eqref{diagtriangle} shows that $\Hr^j(G,K_d)$ vanishes for $j\geq 1$, and the third column of \eqref{diagtriangle} therefore implies that $\Hr^j(G,L_d)$ is nonzero for $j\geq 1$. This is the required contradiction. 
\end{proof}

Let $J$ be the $2$-torsion complex of $G$-modules defined by the distinguished triangle 
\begin{equation}
\label{dist1}
J\to\Zb/2[G]\to\Zb/2[G]\to
\end{equation}
where the morphism $\Zb/2[G]\to\Zb/2[G]$ has image $\Zb/2$. The nonzero cohomology $G$-modules of $J$ are $\Zb/2$ in degrees $0$ and $1$, so one has a distinguished triangle 
\begin{equation}
\label{triangleJ}
\Zb/2\to J\to\Zb/2[-1]\to.
\end{equation}

\subsubsection{Special cases of Lemma \ref{thtriangle}}

\begin{thm}
\label{cortriangle1}
Let $X$ be a smooth variety of dimension $d\geq 1$ over~$\Rb$.
\begin{enumerate}[label=(\roman*)] 
\item 
\label{cori}
If the elements of $\Hr^{d+1}(X(\Cb),\Zb)$ that are in the kernel of both norm maps 
$\Hr^{d+1}(X(\Cb),\Zb)\to \Hr^{d+1}_G(X(\Cb),\Zb)$ and $\Hr^{d+1}(X(\Cb),\Zb)\to \Hr^{d+1}_G(X(\Cb),\Zb(1))$ have coniveau $\geq 2$, then $\ev^d_{\nr,X}$ is surjective.
\item 
\label{corii}
If all classes in the image of $\Hr^{d}_G(X(\Cb),\Zb/2)\to \Hr^{d}(X(\Cb),\Zb/2)$ have co\-ni\-veau~$\geq 2$, then $\oev^{d-1}_{\nr,X}$ is surjective.
\item 
\label{coriii}
For $d\geq 2$, if all classes in the image of the morphism $$\Hr^d_G(X(\Cb),\Zb/2)\to \Hr^d_G(X(\Cb),K_d)$$ induced by \eqref{diagtriangle} have coniveau $\geq 2$, then $\ev^{d-1}_{\nr,X}$ is surjective.
\item 
\label{coriv}
For $d\geq 2$, if all classes in the image of the morphism $$\Hr^d_G(X(\Cb),\Zb/2)\to \Hr^d_G(X(\Cb),J)$$ induced by \eqref{triangleJ} have coniveau~$\geq 2$, then $\oev^{d-2}_{\nr,X}$ is surjective.
\end{enumerate}
\end{thm}

\begin{proof}
In each of the four cases, we apply Lemma \ref{thtriangle} with $j=d$ and $M=\Zb/2$ to an appropriate distinguished triangle $\Zb/2\xrightarrow{a} N\xrightarrow{b} P\xrightarrow{c}$\,. In all cases, the induced morphism $\Hr^{d-1}(G,P)\xrightarrow{c_*} \Hr^d(G,\Zb/2)=\Zb/2$ is bijective. One therefore deduces from the conclusion of Lemma \ref{thtriangle} and from the surjectivity of $\oev^d_X$ (see Theorem \ref{thKvH}) that~$\ev_{\nr,X,P}^{d-1}$ is surjective. In each case, this is equivalent to the desired assertion.

The short exact sequence of $G$-modules $0\to\Zb[G]\to\Zb\oplus\Zb(1)\to\Zb/2\to 0$ induces a distinguished triangle $\Zb/2\to\Zb[G][1]\to(\Zb\oplus\Zb(1))[1]\to$\,. Applying Lemma \ref{thtriangle} to it yields \ref{cori}.

To prove \ref{corii}, apply Lemma~\ref{thtriangle} to the distinguished triangle stemming from the short exact sequence $0\to\Zb/2\to\Zb/2[G]\to\Zb/2\to 0$.

Assertion \ref{coriii} is obtained by applying Lemma \ref{thtriangle} to the distinguished triangle $\Zb/2\to K_d\to L_d\to$ appearing in \eqref{diagtriangle}. That the conclusion of Lemma \ref{thtriangle} implies the desired result in this case follows from Lemma \ref{lemLd}.

Finally, to prove \ref{coriv}, apply Lemma \ref{thtriangle} to the distinguished triangle \eqref{triangleJ}. 
\end{proof}

\begin{cor}
\label{cortriangle2}
Let $X$ be a smooth variety of dimension $d$ over~$\Rb$.
\begin{enumerate}[label=(\roman*)] 
\item 
\label{coroi}
For $d\geq 1$, if $\Hr^{d+1}(X(\Cb),\Zb)[2]$ has coniveau $\geq 2$, then $\ev^d_{\nr,X}$ is surjective.
\item 
\label{coroii}
For $d\geq 1$, if $\Hr^{d}(X(\Cb),\Zb/2)$ has coniveau $\geq 2$, then $\oev^{d-1}_{\nr,X}$ is surjective.
\item 
\label{coroiii}
For $d\geq 2$, if $\Hr^{d+1}(X(\Cb),\Zb)[2]=0$ and $\Hr^d(X(\Cb),\Zb)$ has coniveau~$\geq 2$, then $\ev^{d-1}_{\nr,X}$ is surjective.
\item 
\label{coroiv}
For $d\geq 2$, if $\Hr^d(X(\Cb),\Zb/2)=0$ and $\Hr^{d-1}(X(\Cb),\Zb/2)$ has coniveau $\geq 2$, then $\oev^{d-2}_{\nr,X}$ is surjective.
\end{enumerate}
\end{cor}

\begin{proof}
All assertions follow from the respective assertions of Theorem~\ref{cortriangle1}. 

To show that the hypothesis of Corollary~\ref{cortriangle2}\,\ref{coroi} implies that of Theorem~\ref{cortriangle1}\,\ref{cori}, use the short exact sequence $0\to\Zb[G]\to\Zb\oplus\Zb(1)\to\Zb/2\to 0$.

Case \ref{coroii} is immediate.

Suppose that the assumptions of \ref{coroiii} are satisfied. Taking the $G$-equivariant cohomology of the second row of \eqref{diagtriangle} yields an exact sequence
\begin{equation*}
\Hr^d(X(\Cb),\Zb)\to\Hr^d_G(X(\Cb),K_d)\to\Hr^{d+1}(X(\Cb),\Zb).
\end{equation*}
As $\Hr^{d+1}(X(\Cb),\Zb)[2]=0$, any $2$-torsion class in $\Hr^d_G(X(\Cb),K_d)$ lifts to $\Hr^d(X(\Cb),\Zb)$, and hence has coniveau $\geq 2$. Applying Theorem \ref{cortriangle1}\,\ref{coriii} proves \ref{coroiii}.

Assume finally that the hypotheses of \ref{coroiv} hold. Taking $G$-equivariant cohomology in \eqref{dist1}, we get an exact sequence
$$\Hr^{d-1}(X(\Cb),\Zb/2)\to\Hr^d_G(X(\Cb),J)\to\Hr^{d}(X(\Cb),\Zb/2).$$
As $\Hr^{d}(X(\Cb),\Zb/2)=0$ and $\Hr^{d-1}(X(\Cb),\Zb/2)$ has coniveau $\geq 2$, we deduce that all classes in~$\Hr^d_G(X(\Cb),J)$ have coniveau $\geq 2$ and one can apply Theorem \ref{cortriangle1}\,\ref{coriv}.
\end{proof}

\begin{rem}
Some other particular cases of Lemma \ref{thtriangle} are of interest. For instance, applying it to the distinguished triangle associated with the short exact sequence 
$0\to\Zb/2\to\Qb/\Zb[G]\to(\Qb/\Zb\oplus\Qb/\Zb(1))\to 0$
 shows that if $\Hr^d(X(\Cb),\Qb/\Zb)[2]$ has coniveau $\geq 2$, then the evaluation map $\Hr^{d-1}_{\nr}(X,\Qb/\Zb(d))\to \Hr^0(X(\Rb),\Zb/2)$ is surjective.
The study of evaluation maps with $\Qb/\Zb(j)$ coefficients parallels that with $\Zb/2$ or~$\Zb(j)$ coefficients, but we do not pursue it for lack of specific motivation.
\end{rem}

\subsection{Nonsurjective evaluation maps}

We start with an easy lemma in degree~$\leq\nobreak 2$.

\begin{lem}
\label{lemmepetitdegre}
Let $X$ be a smooth variety over $\Rb$. Let $M$ be a $G$-module. The natural morphism $\Hr^j_G(X(\Cb),M)\to \Hr^j_{\nr}(X,M)$ is bijective if $j\leq 1$ and surjective if $j=2$.
\end{lem}

\begin{proof}
This is a consequence of purity; one can also use the Bloch--Ogus spectral sequence and the vanishing of its $E_2^{p,q}$ terms for $p>q$, see e.g.\,\cite[\S5.1]{benoistIntegralHodgeConjecture2020}.
\end{proof}

\begin{cor}
\label{corpetitdegre}
If $X$ is a smooth variety over $\Rb$, then $\oev^j_{\nr,X}$ and~$\oev^j_{X}$ have the same image for $j\in\{0,1,2\}$, and $\ev^j_{\nr,X}$ and $\ev^j_{X}$ have the same image for $j\in\{1,2\}$. 
\end{cor}

The main result of this subsection applies to unramified cohomology in degree $3$.

\begin{thm}
\label{th3bis}
Let $X$ be a smooth variety over $\Rb$.
\begin{enumerate}[label=(\roman*)] 
\item
\label{th3bisi}
 Suppose that any class~${\alpha\in \Hr^4_G(X(\Cb),\Zb(4))}$ of coniveau $\geq 2$ that is of the form $2\beta$ for some ${\beta\in \Hr^4_G(X(\Cb),\Zb(4))}$ is in fact of the form~$2\beta'$ for some~${\beta'\in\ker(\ev^4_X)}$. Then $\oev^3_X$ and~$\oev^3_{\nr,X}$ have the same image.
 \item 
\label{th3bisii}
Suppose that any class $\alpha\in \Hr^4(X(\Cb),\Zb)$ of coniveau $\geq 2$ that lifts to a class in~${\Hr^4_G(X(\Cb),\Zb(4))}$ in fact lifts to a class in $\ker(\ev^4_X)$. Then $\ev^3_X$ and~$\ev^3_{\nr,X}$ have the same image.
\end{enumerate}
\end{thm}

\begin{proof}
We treat both assertions simultaneously, using the short exact sequence
\begin{equation}
\label{sosNP}
0\to\Zb(4)\xrightarrow{a} N\to P\to 0,
\end{equation}
where $(N,P)=(\Zb(4),\Zb/2)$ in case \ref{th3bisi}, and $(N,P)=(\Zb[G],\Zb(3))$ in case \ref{th3bisii} (in the latter case, the exact sequence \eqref{sosNP} is a special case of \eqref{rcses}). The image of $\ev^3_{X,P}$ is obviously contained in that of $\ev^3_{\nr,X,P}$, and we prove the converse.

Fix $\gamma\in\Hr^3_{\nr}(X,P)$. The unramified cohomology group $\Hr^3_{\nr}(X,P)$ can be described as the subgroup of $\colim_V \Hr^{3}_G(V(\Cb),P)$ (where $V$ runs over all dense open subsets of $X$) consisting of those classes whose residues along all codimension $1$ integral subvarieties of $X$ vanish (see \cite[(5.1) and (5.2)]{benoistIntegralHodgeConjecture2020}). It follows that there exists a closed subset $Z\subset X$ of codimension $\geq 2$, with complement $U:=X\setminus Z$, such that~$\gamma$ is represented by a class $\delta\in\Hr^3_G(U(\Cb),P)$.

Denote by $\partial$ the boundary maps of \eqref{sosNP}. Choose a closed subset $Z'\subset Z$ of codimension $\geq 3$ in $X$ such that $Y:=Z\setminus Z'$ is smooth of pure codimension $2$ in~$U':=X\setminus Z'$. In view of equivariant purity (see \cite[(1.21)]{benoistIntegralHodgeConjecture2020}), the long exact sequence of cohomology with support in $Y$ reads
\begin{equation}
\label{lescohosupport}
\Hr^4_G(U'(\Cb),\Zb(4))\to \Hr^4_G(U(\Cb),\Zb(4))\to \Hr^1_G(Y(\Cb),\Zb).
\end{equation}
As $\partial(\delta)$ is killed by $2$ (this is obvious in case \ref{th3bisi} and follows from~\eqref{rcles} in case~\ref{th3bisii}), its image in the torsion free group $\Hr^1_G(Y(\Cb),\Zb)$ vanishes, so \eqref{lescohosupport} shows that $\partial(\delta)$ lifts to $\Hr^4_G(U'(\Cb),\Zb(4))$. Iterating the argument (by choosing a closed subset~$Z''\subset Z'$ of codimension $\geq 4$ in $X$ such that $Z'\setminus Z''$ is smooth of pure codimension~$3$ in~$X\setminus Z''$, etc.) implies that $\partial(\delta)\in\Hr^4_G(U(\Cb),\Zb(4))$ lifts to a class~$\beta\in\Hr^4_G(X(\Cb),\Zb(4))$.

As $(a_*\beta)|_U=a_*\partial(\delta)=0$, the class $a_*\beta$ has coniveau $\geq 2$. The hypothesis (applied with $\alpha=a_*\beta$) therefore gives $\beta'\in\ker(\ev_X^4)$ with $a_*\beta'=a_*\beta$. By exactness of the long exact sequence attached to \eqref{sosNP}, we deduce the existence of $\varepsilon \in\Hr^3_G(X(\Cb),P)$ such that $\partial(\varepsilon)=\beta-\beta'$. Restricting to $U$ yields $\partial(\varepsilon)|_U=\partial(\delta)-\beta'|_U$. Since the boundary map $\partial$ induces an isomorphism $\Hr^3(G,P)\isoto\Hr^4(G,\Zb(4))$, we deduce that $\ev^3_{X,P}(\varepsilon)|_{U(\Rb)}=\ev^3_{U,P}(\delta)=\ev^3_{\nr,X,P}(\gamma)|_{U(\Rb)}$. As $U(\Rb)$ is dense in~$X(\Rb)$ by smoothness of $X$, we get~$\ev^3_{X,P}(\varepsilon)=\ev^3_{\nr,X,P}(\gamma)$, which completes the proof.
\end{proof}

We will use Theorem \ref{th3bis}\,\ref{th3bisi} twice through the following corollary.

\begin{cor}
\label{corth3bis}
Let $X$ be a smooth proper variety over $\Rb$. Suppose that any Hodge class in $\Hr^4(X(\Cb),\Zb)$
lifts to $\ker(\ev_X^4)$. Then~$\oev_X^3$ and $\oev_{\nr,X}^3$ have the same image.
\end{cor}

\begin{proof}
We verify the hypothesis of Theorem~\ref{th3bis}\,\ref{th3bisi}. In doing so, we identify the canonically isomorphic $G$-modules $\Zb(2)$ and $\Zb(4)$. Let $\pi:X_{\Cb}\to X$ be the base change morphism. Fix $\alpha\in\Hr^4_G(X(\Cb),\Zb(2))$ of coniveau $\geq2$, hence algebraic. Assume that $\alpha=2\beta$ for some $\beta\in\Hr^4_G(X(\Cb),\Zb(2))$. Since $\pi^*\alpha$ is algebraic, it is Hodge, and hence so is $\pi^*\beta$. By hypothesis, there exists $\beta'\in\ker(\ev_X^4)$ with~${\pi^*\beta'=\pi^*\beta}$.
Applying $\pi_*$ to this identity shows that $2\beta'=2\beta$, hence that~$\alpha=2\beta'$.
\end{proof}

\begin{rem}
\label{remH4nr}
Theorem \ref{th3bis} and its proof exploit the deep relation between degree~$3$ unramified cohomology classes and the coniveau $\geq 2$ property for degree $4$ cohomology classes (which is, for instance, used fruitfully in \cite{colliot-theleneCohomologieNonRamifiee2012}).
To prove an analogue of Theorem \ref{th3bis} for degree $4$ unramified cohomology classes, one might similarly try to exploit their relation to Griffiths groups of codimension $3$ cycles (see~\cite{voisinDegree4Unramified2012}). Unfortunately, we could not find a criterion that we are able to apply to any real algebraic variety of interest.
 \end{rem}

\subsection{Examples}

We now apply the results of this section to concrete varieties.

\subsubsection{Products with Enriques surfaces}

\begin{thm}
Let $S$ be an Enriques surface over $\Rb$ such that one half of $S$ is empty and such that $S(\Rb)$ has a connected component $\Pi_S$ of odd Euler characteristic. Let $E$ be a very general elliptic curve over $\Rb$ such that $E(\Rb)$ is not connected. Fix~$d\in\{2,3\}$ and set~$X:=S\times E^{d-2}$. Then $\ev^d_{\nr,X}$ is not surjective.
\end{thm}

\begin{proof}
When $d=2$, combine Theorem \ref{thm:topological_argument_integral_case} and Corollary \ref{corpetitdegre}.

Suppose now that $d=3$. Then $\ev^3_X$ is not surjective by Theorem~\ref{thm:topological_argument_integral_case}. To conclude the proof of the theorem, we check the hypothesis of Theorem \ref{th3bis}\,\ref{th3bisii}. In doing so, it is convenient to identify the canonically isomorphic $G$\nobreakdash-modules $\Zb$, $\Zb(2)$ and $\Zb(4)$. 

Fix ${\alpha\in \Hr^4_G(X(\Cb),\Zb(2))}$ whose image~$\beta\in \Hr^4(X(\Cb),\Zb)$ has coniveau $\geq 2$, and hence is algebraic. As $\Hr^*(E(\Cb),\Zb)$ is torsion free, the K\"unneth decomposition reads
\begin{equation}
\label{Kunneth}
\Hr^4(X(\Cb),\Zb)=\bigoplus_{i=2}^4\Hr^{i}(S(\Cb),\Zb)\otimes_\Zb\Hr^{4-i}(E(\Cb),\Zb).
\end{equation}
Let $(\beta_2,\beta_3,\beta_4)$ be the image of $\beta$ in \eqref{Kunneth}. As $E$ is very general, no nonzero class in $\Hr^{3}(S(\Cb),\Zb)\otimes_\Zb\Hr^{1}(E(\Cb),\Zb)\simeq \Hr^1(E(\Cb),\Zb/2)$ is algebraic (see \cite[Theorem~0.1]{benoistOttemFailureIHC2020}). To be precise, this reference only states that one of these classes is not algebraic; this implies our stronger assertion because the monodromy of the universal family of complex elliptic curves acts transitively on the nonzero classes in~$\Hr^1(E(\Cb),\Zb/2)$. Since the two other factors of \eqref{Kunneth} are entirely algebraic, one has~$\beta_3=0$.

Since $\beta=\beta_2+\beta_4$ lifts to $\Hr^4_G(X(\Cb),\Zb(2))$, it belongs to the $G$-invariant part of 
\begin{equation*}
\Hr^{2}(S(\Cb),\Zb(1))\otimes_\Zb\Hr^{2}(E(\Cb),\Zb(1))\oplus\Hr^{4}(S(\Cb),\Zb(2))\otimes_\Zb\Hr^{0}(E(\Cb),\Zb).
\end{equation*}
As the groups $\Hr^{0}(E(\Cb),\Zb)$ and $\Hr^{2}(E(\Cb),\Zb(1))$, as well as $\Hr^{4}(S(\Cb),\Zb(2))$, carry a trivial action of $G$, the class $\beta=\beta_2+\beta_4$ in fact belongs to 
\begin{equation}
\label{Kunneth3}
\Hr^{2}(S(\Cb),\Zb(1))^G\otimes_\Zb\Hr^{2}(E(\Cb),\Zb(1))\oplus\Hr^{4}(S(\Cb),\Zb(2))\otimes_\Zb\Hr^{0}(E(\Cb),\Zb).
\end{equation}
The four groups appearing in \eqref{Kunneth3} are spanned by classes of algebraic cycles defined over $\Rb$. Indeed, the groups $\Hr^{2}(E(\Cb),\Zb(1))=\Zb$ and ${\Hr^{4}(S(\Cb),\Zb(2))=\Zb}$ are generated by the class of a point (and $E$ and $S$ have real points), and $\Hr^{0}(E(\Cb),\Zb)$ is generated by the class of $E$. As for $\Hr^{2}(S(\Cb),\Zb(1))^G$, consider the composition
$$\Pic(S)\to \Hr^{2}_G(S(\Cb),\Zb(1))\to \Hr^{2}(S(\Cb),\Zb(1))^G$$
of Krasnov's equivariant cycle class map (for which see \cite{krasnovEquivariantGrothendieckCohomology1994} or \cite[\S 1.6.1]{benoistIntegralHodgeConjecture2020}) and of the natural morphism and note that the former is surjective by \cite[Proposition 2.8]{benoistIntegralHodgeConjecture2020} as~$\Hr^2(S,\Oc_S)=0$, and the latter is surjective by the Hochschild--Serre spectral sequence (see \cite[(1.4)]{benoistIntegralHodgeConjecture2020}) as $\Hr^1(S(\Cb),\Zb)=0$ and $S(\Rb)\neq\varnothing$.

This proves that $\beta\in\Hr^{4}(X(\Cb),\Zb(2))$ lifts to a class $\gamma\in\Hr^{4}_G(X(\Cb),\Zb(2))$ that is algebraic (in the sense that it is in the image of the equivariant cycle class map). Being algebraic, the class $\gamma$ vanishes when restricted to any real point of~$X$.
\end{proof}

Arguing as in the proof of Corollary \ref{corevd}, we deduce the following.

\begin{cor}
\label{corod=23}
For $d\in\{2,3\}$, there exists a smooth projective variety $X$ of dimension~$d$ over~$\Rb$ such that $\ev^d_{\nr,X}$ is not surjective.
\end{cor}

Corollary \ref{corod=23} cannot hold for $d=1$ (see Corollary \ref{corod=1}). We do not know how to prove it for~$d\geq 4$ (see Question \ref{qnotsurj}\,\ref{qi}).

\subsubsection{Complete intersections of two quadrics}

\begin{prop}
\label{propciquadrics}
Let $X$ be a smooth complete intersection of two quadrics of dimension $d$ over~$\Rb$ such that $X(\Rb)$ is not connected. If $d\in \{1,2,3\}$, then $\oev^{d-1}_{\nr,X}$ is not surjective. If~$d\geq 4$, then $\ev^{d-1}_{\nr,X}$ is surjective.
\end{prop}

\begin{proof}
Suppose first that $d\in \{1,2,3\}$. In view of Corollary \ref{corpetitdegre}, it suffices to show that $\oev^{d-1}_{X}$ is not surjective, which is a consequence of Proposition \ref{prop:image_global_sign_hypersurface}.

When $d\geq 4$, we verify the two hypotheses of Corollary \ref{cortriangle2}\,\ref{coroiii}. First, one has~$\Hr^{d+1}(X(\Cb),\Zb)[2]=0$ by the weak Lefschetz theorem and Poincar\'e duality. Second, the cohomology group $\Hr^d(X(\Cb),\Zb)$ has coniveau $\frac{d}{2}\geq 2$ if $d$ is even by \cite[Corollary 3.15]{reidCompleteIntersectionQuadrics1972} and coniveau $\frac{d-1}{2}\geq 2$ if $d$ is odd by \cite[Theorem~4.14\,(b) p.\,73]{reidCompleteIntersectionQuadrics1972}. In more detail, when $d$ is odd, Reid constructs a smooth projective curve $C$ over $\Cb$ and an isomorphism~$\Hr^1(C(\Cb),\Zb)\isoto \Hr^{d}(X(\Cb),\Zb)$ (called \emph{cylinder homomorphism}) induced by an algebraic correspondence. Since this correspondence is given by a cycle of dimension $\frac{d+1}{2}$, the image of this isomorphism has coniveau~$\frac{d-1}{2}\geq 2$.
\end{proof}

\subsubsection{Cubic hypersurfaces}

\begin{prop}
\label{propcubic}
Let $X$ be a smooth cubic hypersurface of dimension $d$ over~$\Rb$ such that~$X(\Rb)$ is not connected. If $d\in\{1,2,3\}$, or if $d=4$ and $X$ is very general,
then~$\oev^{d-1}_{\nr,X}$ is not surjective. If~$d\geq 5$, then $\ev^{d-1}_{\nr,X}$ is surjective.
\end{prop}

\begin{proof}
For  $d\in \{1,2,3\}$, combine Proposition \ref{prop:image_global_sign_hypersurface} and Corollary \ref{corpetitdegre}.

Suppose that $d=4$ and that $X$ is very general. By Proposition~\ref{prop:image_global_sign_hypersurface}, the morphism~$\oev^3_{X}$ is not surjective. To conclude the proof in this case, it therefore suffices to verify the hypothesis of Corollary \ref{corth3bis}. Let $\pi:X_{\Cb}\to X$ be the base change morphism, let $h\in \Hr^2_G(X(\Cb),\Zb(1))$ be the hyperplane class in $G$\nobreakdash-equivariant cohomology, and identify $\Zb(2)$ and $\Zb(4)$. Let $\alpha\in \Hr^4(X(\Cb),\Zb)$ be a Hodge class. Since~$X$ is very general, all Hodge classes in $\Hr^4(X(\Cb),\Zb)$ are multiples of $(\pi^*h)^2$ by \cite[Chap.\,1, Proof of Corollary 2.12]{huybrechtsGeometryCubicHypersurfaces2023} and primitivity of $(\pi^*h)^2$ in $\Hr^4(X(\Cb),\Zb)$. It follows that $\alpha=\ell (\pi^*h)^2=\pi^*(\ell h^2)$ for some $\ell\in\Zb$. This concludes because, the class $h^2$ being algebraic, one has $\ev^4_X(h^2)=0$.

For $d\geq 5$, apply Corollary \ref{cortriangle2}\,\ref{coroiii}, whose hypotheses are satisfied by the weak Lefschetz theorem and Poincar\'e duality, and by Lemma \ref{lemconiveaucubic} below.
\end{proof}

The following lemma is easy for $d\leq 4$ and appears in \cite[Proposition 10\,(2)]{fu2CyclesHypersurfacesCubiques2019} for $d=5$. We deduce it in general from the work of Lewis \cite{lewisCylinderHomomorphisms1993}.

\begin{lem}
\label{lemconiveaucubic}
Let $X$ be a smooth cubic hypersurface of dimension $d\geq 1$ over $\Cb$. Then $\Hr^d(X(\Cb),\Zb)$ has coniveau $\geq \lfloor\sqrt{6d+37}-6\rfloor$.
\end{lem}

\begin{proof}
Assume first that $X$ is general. Set $k:= \lfloor\sqrt{6d+37}-6\rfloor$. This is the largest integer $k$ satisfying 
\cite[(2.3)]{lewisCylinderHomomorphisms1993} (note that in \cite{lewisCylinderHomomorphisms1993}, the integers $n$ and $d$ denote the dimension and the degree of $X$). We apply the results of \cite{lewisCylinderHomomorphisms1993} with this value of $k$ (this is legitimate as indicated in the note added in proof in \cite[p.\,220]{lewisCylinderHomomorphisms1993}).

Lewis constructs a smooth projective variety $\Omega_X$ of dimension $d-2k$ over~$\Cb$ and a cylinder homomorphism $\Phi_*:\Hr^{d-2k}(\Omega_X(\Cb),\Zb)\to \Hr^d(X(\Cb),\Zb)$ induced by an algebraic correspondence, so the image of $\Phi_*$ has coniveau $\geq k$ (see~\mbox{\cite[(3.7)]{lewisCylinderHomomorphisms1993}}). The lemma follows since $\Hr^d(X(\Cb),\Zb)$ is equal to $\im(\Phi_*)$ if $d$ is odd, and is generated by $\im(\Phi_*)$ and a power of the hyperplane class if $d$ is even (see \cite[p.\,213]{lewisCylinderHomomorphisms1993}).

The result for arbitrary $X$ follows from the case of a general $X$ by a specialization argument (exactly as in the second part of \cite[Proof of Proposition 10]{fu2CyclesHypersurfacesCubiques2019}).
\end{proof}

\subsubsection{Abelian varieties}

\begin{prop}
\label{propva2}
Let $X$ be an abelian variety of dimension $d$ over $\Rb$ such that~$X(\Rb)$ has $2^d$ connected components. If $d\in\{1,2,3\}$, or if $d=4$ and $X$ is principally polarized and very general, then $\oev^{d-1}_{\nr,X}$ is not surjective.
\end{prop}

\begin{proof}
By Proposition \ref{propva}, the morphism $\oev^{d-1}_{X}$ is not surjective. To conclude, we show that the images of $\oev^{d-1}_{X}$ and $\oev^{d-1}_{\nr,X}$ coincide. When~${d\in\{1,2,3\}}$, this follows from Corollary \ref{corpetitdegre}. Assuming from now on that $d=4$ and that~$X$ is principally polarized and very general, it remains to verify the hypothesis of Corollary~\ref{corth3bis}.

Let $\pi:X_{\Cb}\to X$ be the base change morphism, let $\theta\in \Hr^2_G(X(\Cb),\Zb(1))$ be the class of the principal polarization in $G$\nobreakdash-equivariant cohomology, and identify $\Zb(2)$ and~$\Zb(4)$. Fix a Hodge class $\alpha\in \Hr^4(X(\Cb),\Zb)$. As the abelian variety~$X$ is principally polarized and very general, Mattuck's theorem (see \cite[Theorem~17.4.1]{birkenhakeLangeComplexAbelianVarieties2004}) and the fact that~$\frac{(\pi^*\theta)^2}{2}\in\Hr^4(X(\Cb),\Zb)$ is primitive imply that $\alpha=\ell\frac{(\pi^*\theta)^2}{2}$ for some $\ell\in\Zb$. By Lemma \ref{lemvapp} below applied with $d=4$ and $c=2$, there exists~$\eta\in\ker(\ev^4_X)$ such that $\pi^*\eta=\frac{(\pi^*\theta)^2}{2}$. It follows that $\alpha=\pi^*(\ell\eta)$, which concludes the proof.
\end{proof}

\begin{lem}
\label{lemvapp}
Let $X$ be a principally polarized abelian variety of dimension~$d$ over~$\Rb$ such that $X(\Rb)$ has $2^d$ connected components. Let $\theta\in \Hr^2_G(X(\Cb),\Zb(1))$ be the class of the principal polarization. Fix $1\leq c\leq d$. There exists $\eta\in\Hr^{2c}_G(X(\Cb),\Zb(c))$ vanishing in restriction to each real point of $X$ such that the image of $\theta^c-c!\,\eta$ in~$\Hr^{2c}(X(\Cb),\Zb)$~vanishes.
\end{lem}

\begin{proof}
Let $E$ be an elliptic curve over $\Rb$ such that $E(\Rb)$ has $2$ connected components. By the connectedness of the parameter space for principally polarized abelian varieties $X$ of dimension $d$ over $\Rb$ such that $X(\Rb)$ has $2^d$ connected components (see \cite[Proposition 9.3]{grossHarrisRealAlgebraicCurves1981} or \cite[IV, Definition 4.4, Corollary 4.5, Theorem 4.6]{silholRealAlgebraicSurfaces1989}), and since the statement is purely topological, we may assume that $X=E^d$ and that the principal polarization is induced by principal polarizations of the factors.

Let $Z\subset X$ be the union of the $\binom{d}{d-c}$ abelian subvarieties of $X=E^d$ that are products of~$d-c$ of the factors. Let $\eta\in\Hr^{2c}_G(X(\Cb),\Zb(c))$ be the image of $Z$ by the $G$\nobreakdash-equivariant cycle class map (for which see \cite{krasnovEquivariantGrothendieckCohomology1994} or \cite[\S 1.6.1]{benoistIntegralHodgeConjecture2020}). The class~$\eta$ vanishes upon restriction to the real points of $X$ because it is algebraic, and the assertion concerning its image in $\Hr^{2c}(X(\Cb),\Zb)$ is a direct computation.
\end{proof}

 \section{\texorpdfstring{$\Ibf$}{I}-cohomology}

\subsection{Bounding the exponent of the cokernel of higher signature morphisms}
\label{parexponentbounds}

In this subsection, we prove Theorem \ref{thexpcokerintro} from the introduction.

\subsubsection{Canonical sections of evaluation morphisms}

We first introduce some notation. Fix~$j\geq 0$. Let~$X$ be a smooth algebraic variety over~$\Rb$ and let~$\iota :X(\Rb)\to X$ be the inclusion. The natural morphisms $\Hr^c(X,\iota_*\Lb)\to \Hr^c(X(\Rb),\Lb)$ are isomorphisms for all $c\geq 0$ and all local systems $\Lb$ on $X(\Rb)$ (see e.g.\,\cite[Proposition~2.26]{lerbetImageHigherSignature2026}).

The maps $\overline{\ev}_{U}^j$ for~${U\subset X}$ open assemble into a morphism $\overline{\ev}^j$ of presheaves on $X$, from the presheaf ${U\mapsto\Hr_G^j(U(\Cb),\Zb/2)}$ to the sheaf $\iota_*\Zb/2$ given by the assignment~$U\mapsto\Hr^0(U(\Rb),\Zb/2)$. Upon sheafification, this induces a homomorphism~${\overline{\ev}_\nr^j:\Hsc_X^j(\Zb/2)\to\iota_*\Zb/2}$ that recovers the map~$\overline{\ev}_{\nr,X}^j$ after taking global sections. We record our crucial observation as follows.

\begin{thm}
\label{thm:functorial_Krasnov}
Let $X$ be a smooth algebraic variety of dimension $d$ over $\Rb$. The homomorphism of sheaves $\overline{\ev}_\nr^d:\Hsc_X^d(\Zb/2)\to\iota_*\Zb/2$  has a canonical section.
\end{thm}

\begin{proof}
If $U$ is an open subset of $X$ and $V$ is a connected component of $U(\Rb)$, we let $\overline{\sigma}_V\in\Hr_G^d(U(\Cb),\Zb/2)$ be the equivariant fundamental class of $V$ \cite[p. 15]{benoistIntegralHodgeConjecture2020}. For fixed $U\subseteq X$, the assignment $V\mapsto\overline{\sigma}_V$ extends to a group morphism 
\[\overline{\sigma}(U):\Hr^0(U(\Rb),\Zb/2)\to\Hr_G^d(U(\Cb),\Zb/2).\] 
 By \cite[Remark 1.6 (ii)]{benoistIntegralHodgeConjecture2020}, one has $\overline{\ev}_U^d\circ\overline{\sigma}(U)=\Id$. The definition of the equivariant fundamental class shows that the morphisms $\overline{\sigma}(U)$ are natural in~$U$, hence determine a morphism $\overline{\sigma}$ of presheaves from $\iota_*\Zb/2$ to the presheaf~${U\mapsto\Hr_G^d(U(\Cb),\Zb/2)}$. Let~$\overline{\sigma}_\nr:\iota_*\Zb/2\to\Hsc_X^d(\Zb/2)$ be the sheafification of~$\overline{\sigma}$. As sheafification is a functor, the equality $\overline{\ev}^d\circ\overline{\sigma}=\Id$ of endomorphisms of $\iota_*\Zb/2$ implies that $\overline{\ev}_\nr^d\circ\overline{\sigma}_\nr=\Id$. The morphism $\overline{\sigma}_\nr$ is then the required section.
\end{proof}

The following corollary generalizes \cite[Theorem 3.2 (d)]{colliot-theleneZerocyclesCohomologyReal1996} (which deals with the cases $c\in\{d,d-1\}$), and \cite[Corollary 3.3 (b)]{colliot-theleneZerocyclesCohomologyReal1996} (which proves the case $c=d-2$ under the additional hypothesis that $\Hr^{2d-1}(X(\Cb),\Zb/2)=0$).

\begin{cor}
\label{cor:mod_2_cycle_class_map_surjective}
Let~$X$ be a smooth algebraic variety of dimension~$d$ over $\Rb$. For all~${c\geq 0}$, the morphism $\Hr^c(X,\Hsc_X^d(\Zb/2))\to\Hr^c(X(\Rb),\Zb/2)$ obtained from~${\overline{\ev}_\nr^d}$ by taking cohomology in degree $c$ admits a canonical section.
\end{cor}

\begin{proof}
This follows from the fact that ${\overline{\ev}_\nr^d}$ has a canonical section (see Theorem~\ref{thm:functorial_Krasnov}), by functoriality of cohomology.
\end{proof}

\subsubsection{Canonical sections of signature morphisms}

 Let $X$ be a smooth algebraic variety of dimension~$d$ over~$\Rb$, let~$\mathcal{L}$ be a line bundle on $X$, and let $n\geq 0$. As indicated in §\ref{parWitt}, the signature homomorphism can be twisted by~$\Lc$ into a morphism $\sign_n:\Ibf^n(\Lc)\to\iota_*\Zb(\Lc)$ of sheaves on~$X$ (see \eqref{signature5}). As a consequence of the Milnor conjecture proved in \cite{orlovExactSequenceKM22007}, the inclusion~$\Ibf^{n+1}(\Lc)\subseteq\Ibf^n(\Lc)$ of powers of the fundamental ideal fits into an exact sequence
\begin{equation}
\label{InHn}
0\to\Ibf^{n+1}(\Lc)\to\Ibf^n(\Lc)\xrightarrow{\pi_n}\Hsc_X^n(\Zb/2)\to 0
\end{equation}
of sheaves on~$X$ (see the isomorphism $\overline{\Ibf}^*\simeq\Hbf^*$ in \cite[p.\,185]{lerbetImageHigherSignature2026}; this reference uses the sheafification of~$U\mapsto\Hr_\et^n(U,\Zb/2)$, which coincides with our definition of~$\Hsc_X^n(\Zb/2)$ in view of Cox's comparison theorem between \'etale and $G$-equivariant cohomology, see \cite{coxEtaleHomotopyType1979}). As explained in \cite[pp.\,189--190]{lerbetImageHigherSignature2026}, the signature and evaluation homomorphisms make \eqref{InHn} part of a commutative exact ladder
\begin{equation}
\label{eq:milnor_conjecture}
\begin{tikzcd}[row sep=small]
	0 & {\Ibf^{n+1}(\Lc)} & {\Ibf^n(\Lc)} & {\Hsc_X^n(\Zb/2)} & 0 \\
	0 & {\iota_*\Zb(\Lc)} & {\iota_*\Zb(\Lc)} & {\iota_*\Zb/2} & 0
	\arrow[from=1-1, to=1-2]
	\arrow[from=1-2, to=1-3]
	\arrow["{\sign_{n+1}}"', from=1-2, to=2-2]
	\arrow["{\pi_n}", from=1-3, to=1-4]
	\arrow["{\sign_n}"', from=1-3, to=2-3]
	\arrow[from=1-4, to=1-5]
	\arrow["{\oev_{\nr}^n}"', from=1-4, to=2-4]
	\arrow[from=2-1, to=2-2]
	\arrow["{\cdot 2}"', from=2-2, to=2-3]
	\arrow["{\text{mod}\;2}"', from=2-3, to=2-4]
	\arrow[from=2-4, to=2-5]
\end{tikzcd}
\end{equation}
of sheaves on~$X$. When $n=d$, the left vertical arrow in \eqref{eq:milnor_conjecture} is an isomorphism (see \cite[Corollary 2.33]{lerbetImageHigherSignature2026}), so the commutative diagram
\begin{equation}
\label{eq:milnor_conjecture2}
\begin{tikzcd}[row sep=small]
	 {\Ibf^d(\Lc)} & {\Hsc_X^d(\Zb/2)}  \\
	 {\iota_*\Zb(\Lc)} & {\iota_*\Zb/2} 
	\arrow["{\pi_d}",from=1-1, to=1-2]
	\arrow["{\oev_{\nr}^d}"', from=1-2, to=2-2]
	\arrow["{\sign_d}"', from=1-1, to=2-1]
	\arrow["{\text{mod}\;2}"', from=2-1, to=2-2]
\end{tikzcd}
\end{equation}
is a fiber product square.

The following corollary of Theorem \ref{thm:functorial_Krasnov} generalizes results of Hornbostel \cite{hornbostelFewComputationsReal2025} and Lerbet~\cite{lerbetImageHigherSignature2026}, who proved the surjectivity of $\sign_{d,X}^c$ in a number of cases (when~${c\in\{d-1,d\}}$ in \cite[Proposition 3.4, Proof of Lemma 5.1]{hornbostelFewComputationsReal2025} and when ${c\in\{0,d-1,d\}}$ or when~$c=d-2$ and $\Hr^{2d-1}(X(\Cb),\Zb/2)=0$ in \cite[Theorem~3.5, Proof of Proposition~4.4, Corollary~4.8 and Corollary~5.2]{lerbetImageHigherSignature2026}). 

\begin{cor}
\label{cor:integral_cycle_class_map_split_surjective}
Let $X$ be a smooth variety of dimension~$d$  over $\Rb$ and let~$\Lc$ be a line bundle on~$X$. The morphism $\sign_d:\Ibf^d(\Lc)\to\iota_*\Zb(\Lc)$ has a canonical section, which induces a canonical section of  $\sign_{d,X}^c:\Hr^c(X,\Ibf^d(\Lc))\to\Hr^c(X(\Rb),\Zb(\Lc))$ for~$c\geq 0$.
\end{cor}

\begin{proof}
By pullback along the reduction mod~$2$ morphism $\iota_*\Zb(\Lc)\to\iota_*\Zb/2$, the section $\overline{\sigma}_\nr$ of $\overline{\ev}_\nr^d$ constructed in the proof of Theorem \ref{thm:functorial_Krasnov} induces a section
\begin{equation}
\label{sectiontau}
\iota_*\Zb(\Lc)\to\Hsc_X^d(\Zb/2)\times_{(\iota_*\Zb/2)}\iota_*\Zb(\Lc)\simeq\Ibf^d(\Lc)
\end{equation}
of $\sign_d$ (where the isomorphism in \eqref{sectiontau} stems from the fiber product square \eqref{eq:milnor_conjecture2}). By functoriality of cohomology, the morphism $\sign_{d,X}^c$ obtained from $\sign_d$ by taking cohomology also has a section.
\end{proof}

\subsubsection{Exponent bounds}

The following theorem (Theorem \ref{thexpcokerintro} from the introduction) confirms the validity of the exponent bound predicted in \cite[Conjecture~6.5]{lerbetImageHigherSignature2026}.

\begin{thm}
\label{thexpcoker}
Let $X$ be a smooth variety of dimension~$d$ over $\Rb$, let~$\Lc$ be a line bundle on~$X$, and let $0\leq c\leq d$ and $0\leq n\leq d$ be integers. Then the cokernel of the homomorphism $\sign_{n,X}^c:\Hr^c(X,\Ibf^n(\Lc))\to\Hr^c(X(\Rb),\Zb(\Lc))$ is killed by $2^{d-n}$.
\end{thm}

\begin{proof}
Since $n\leq d$, the sheaf $\Ibf^d(\Lc)$ is a subsheaf of $\Ibf^n(\Lc)$. By definition of the signature homomorphisms (see \cite[(2-4)]{lerbetImageHigherSignature2026}), the inclusion $\Ibf^d(\Lc)\subseteq\Ibf^n(\Lc)$ induces a commutative square
\begin{equation}
\label{signdn}
\begin{tikzcd}[row sep=small]
	{\Hr^c(X,\Ibf^d(\Lc))} & {\Hr^c(X,\Ibf^n(\Lc))} \\
	{\Hr^c(X(\Rb),\Zb(\Lc))} & {\Hr^c(X(\Rb),\Zb(\Lc)).}
	\arrow[from=1-1, to=1-2]
	\arrow["{\sign_{d,X}^c}"', from=1-1, to=2-1]
	\arrow["{\,\sign_{n,X}^c}", from=1-2, to=2-2]
	\arrow["{\cdot 2^{d-n}}"', from=2-1, to=2-2]
\end{tikzcd}
\end{equation}
Since $\sign_{d,X}^c$ is surjective by Corollary \ref{cor:integral_cycle_class_map_split_surjective}, the square \eqref{signdn} implies that the image of $\sign_{n,X}^c$ contains the multiples of $2^{d-n}$. This proves the theorem.
\end{proof}

\begin{rem}
In \cite{lerbetWittGroupsSmooth2026}, the second author studies Witt groups of smooth real algebraic varieties of small dimension. The results of \S\ref{parexponentbounds} can be used to improve or simplify some of the arguments there. As an example, the surjectivity hypothesis in~\cite[Lemma 2.14]{lerbetWittGroupsSmooth2026} is automatic by Corollary \ref{cor:mod_2_cycle_class_map_surjective}.
\end{rem}

\subsection{On the optimality of the exponent bound}

\subsubsection{Powers of the punctured affine line}

We first study $\Gb_{m,\Rb}^d:=(\Abb^1_{\Rb}\setminus\{0\})^d$. The following proposition and its proof are closely related to \cite[\S5]{hennigRealCycleClass2025}.

\begin{prop}
\label{propmultiplicative}
Fix $d\geq 0$ and $0\leq n<d$. Set $X:=\Gb_{m,\Rb}^d$. The cokernel of the signature morphism $\sign_{n,X}^0:\Hr^0(X,\Ibf^n)\to\Hr^0(X(\Rb),\Zb)$ is not killed by $2^{d-n-1}$.
\end{prop}

\begin{proof}
Let $(t_1,\ldots,t_d)$ be the coordinates of $X$. If $I=\{i_1<\cdots<i_l\}$ is a subset of~$\{1,\ldots,d\}$, we consider the quadratic form \[q_{I,n}:=\varepsilon_{|I|,n}\otimes\langle 1,-t_{i_1}\rangle\otimes\cdots\otimes\langle 1,-t_{i_l}\rangle\in\Hr^0(X,\Ibf^n),\] where, given an integer $0\leq l\leq d$, the form $\varepsilon_{l,n}$ is defined as
\[\varepsilon_{l,n}:=\left\{\begin{array}{cl}
\langle 1,1\rangle^{\otimes n-l} & \text{if}\;l\leq n\\
\langle 1\rangle                 & \text{if}\;l>n.
\end{array}\right.\]

Applying the proof of \cite[Lemma 2.16]{bachmannMilnorWittMotives2025} to $\Ibf$-cohomology instead of $\Kbf^{\text{MW}}$\nobreakdash-cohomology (making use of homotopy invariance \cite[Théorème 11.2.9]{faselGroupesChowWitt2008} and purity \cite[Theorem 2.41]{hornbostelRealCycleClass2021} for $\Ibf$-cohomology) shows that
$$\Hr^0(\Gb_{m,\Rb}^d,\Ibf^n)=\Hr^0(\Gb_{m,\Rb}^{d-1},\Ibf^n)\oplus \Hr^0(\Gb_{m,\Rb}^{d-1},\Ibf^{n-1})\cdot\langle 1,-t_d\rangle,$$
with the convention that $\Ibf^m=\Wbf$ when $m\leq 0$. Induction on $d$ then implies that~$\Hr^0(X,\Ibf^n)$ is the free abelian group generated by the forms $(q_{I,n})_{I\subseteq\{1,\ldots,d\}}$.

Given an open and closed subset $A$ of $X(\Rb)$, we let $\chi_A\in\Hr^0(X(\Rb),\Zb)$ denote its characteristic function. The set $\pi_0(X(\Rb))$ of connected components of $X(\Rb)$ is in bijection with the set of subsets $I$ of $\{1,\ldots,d\}$ via the association $I\mapsto A_I$, where
\[A_I:=\{(t_i)_i\in X(\Rb)\mid t_i<0\;\text{if}\;i\in I\text{ and }t_i>0\;\text{if}\;i\notin I\}.\] 
We conclude that $(\chi_{A_I})_{I\subseteq\{1,\ldots,d\}}$ is a basis of $\Hr^0(X(\Rb),\Zb)=\Zb^{\pi_0(X(\Rb))}$. Set further 
$$B_I:=\{(t_i)_i\in X(\Rb)\mid t_i<0\;\text{if}\;i\in I\}.$$
Then $A_{I}=B_I\setminus\bigcup_{i\notin I}B_{I\cup\{i\}}$. For $\{i_1,\dots, i_k\}=J\subseteq\{1,\dots,d\}\setminus I$, one computes that~$B_{I\cup\{i_1\}}\cap\cdots\cap B_{I\cup\{i_k\}}=B_{I\cup\{i_1,\ldots,i_k\}}$, so the inclusion-exclusion formula yields 
\[\chi_{A_I}=\sum_{k=0}^{d-|I|}(-1)^{k}\sum_{J\subseteq\{1,\ldots,d\}\setminus I,\,|J|=k}\chi_{B_{I\cup J}}.\] 
We deduce that the family $(\chi_{B_I})_{I\subseteq\{1,\ldots,d\}}$ generates $\Hr^0(X(\Rb),\Zb)$. Since its cardinality is the rank $2^d$ of this free abelian group, it follows that $(\chi_{B_I})_{I\subseteq\{1,\ldots,d\}}$ is in fact a basis of $\Hr^0(X(\Rb),\Zb)$. 

For $I\subseteq\{1,\ldots,d\}$, set $u_{I,n}:=\max(2^{|I|-n},1)\in\Zb$. By definition, one has \[\sign_{n,X}^0(q_{I,n})=u_{I,n}\chi_{B_I}.\] We conclude that $\sign_{n,X}^0:\Hr^0(X,\Ibf^n)\to\Hr^0(X(\Rb),\Zb)$ is an isomorphism onto the subgroup of $\Hr^0(X(\Rb),\Zb)$ generated by the $(u_{I,n}\chi_{B_I})_{I\subseteq\{1,\ldots,d\}}$, hence that the cokernel of $\sign_{n,X}^0$ is isomorphic to $\bigoplus_{I\subseteq\{1,\ldots,d\}}\Zb/u_{I,n}$. This group is not killed by~$2^{d-n-1}$ because $u_{\{1,\ldots,d\},n}=2^{d-n}$.
\end{proof}

\subsubsection{Cubic hypersurfaces}

We now focus on the case of smooth cubic hypersurfaces.

\begin{prop}
\label{prop:cokernel_cubic}
Let $X\subseteq\Pb_\Rb^{d+1}$ be a smooth cubic hypersurface of dimension~$d\leq 4$. If $d=4$, we further assume that~$X$ is very general. Fix $0\leq n\leq d$. The image of 
$$\sign_{n,X}^0:\Hr^0(X,\Ibf^n)\to\Hr^0(X(\Rb),\Zb)$$
is generated by $1\in \Hr^0(X(\Rb),\Zb)$ and by the multiples of $2^{d-n}$.

If $n<d$ and $X(\Rb)$ is disconnected, the cokernel of $\sign_{n,X}^0$ is not killed by~$2^{d-n-1}$.
\end{prop}

\begin{proof}
The second assertion is a direct consequence of the first, which we now prove.

We argue by decreasing induction on $n$. When $n=d$, the morphism $\sign_{n,X}^0$ is surjective by \cite[Corollary 5.2]{lerbetImageHigherSignature2026}, which proves the proposition in this case. Suppose from now on that $n<d$. We will rely on the commutative exact~diagram
\begin{equation}
\label{signsignev}
\begin{tikzcd}[row sep=small, column sep=small]
	0 & {\Hr^0(X,\Ibf^{n+1})} & {\Hr^0(X,\Ibf^n)} & {\Hr^0(X,\Hsc_X^n(\Zb/2))=\Hr_\nr^n(X,\Zb/2)} \\
	0 & {\Hr^0(X(\Rb),\Zb)} & {\Hr^0(X(\Rb),\Zb)} & {\Hr^0(X(\Rb),\Zb/2)}
	\arrow[from=1-1, to=1-2]
	\arrow[from=1-2, to=1-3]
	\arrow["{\sign_{n+1,X}^0}"', from=1-2, to=2-2]
	\arrow[from=1-3, to=1-4]
	\arrow["{\sign_{n,X}^0}"', from=1-3, to=2-3]
	\arrow["{\overline{\ev}_{\nr,X}^n}"', from=1-4, to=2-4]
	\arrow[from=2-1, to=2-2]
	\arrow["{\cdot 2}"', from=2-2, to=2-3]
	\arrow[from=2-3, to=2-4]
\end{tikzcd}
\end{equation}
obtained by taking cohomology in (\ref{eq:milnor_conjecture}).

Assume first that $d=4$ (so $X$ is very general) and $n=3$. It follows from the proof of Proposition \ref{propcubic} for $d=4$ (see how Proposition~\ref{prop:image_global_sign_hypersurface} and Corollary~\ref{corth3bis} are combined there) that the image of $\overline{\ev}_{\nr,X}^3$ is generated by~$1\in\Hr^0(X(\Rb),\Zb/2)$. We deduce from \eqref{signsignev} that the image of $\sign_{3,X}^0$ is contained in the subgroup generated by~$1\in \Hr^0(X(\Rb),\Zb)$ and by the multiples of $2$. The class $1\in \Hr^0(X(\Rb),\Zb)$ is in the image of $\sign_{3,X}^0$ because $\sign_{3,X}^0(\langle 1,1\rangle^{\otimes 3})=1$. So are the multiples of~$2$ by \eqref{signsignev} and the surjectivity of $\sign_{4,X}^0$ (the base case of the induction). This completes the proof of the proposition in this case.

We henceforth assume that $n\leq 2$. Restriction and sheafification morphisms fit into a commutative square
\begin{equation}
\label{restrictionsheafification}
\begin{tikzcd}[row sep=small, column sep=small]
	{\Hr_G^n(\Pb^{d+1}(\Cb),\Zb/2)} & {\Hr_G^n(X(\Cb),\Zb/2)} \\
	{\Hr_\nr^n(\Pb_\Rb^{d+1},\Zb/2)} & {\Hr_\nr^n(X,\Zb/2)}
	\arrow["\sim", from=1-1, to=1-2]
	\arrow[from=1-1, to=2-1]
	\arrow[two heads, from=1-2, to=2-2]
	\arrow[from=2-1, to=2-2]
\end{tikzcd}
\end{equation}
of abelian groups, whose top horizontal arrow is bijective by Proposition~\ref{prop:weak_lefschetz} (as~${n<d}$) and whose right vertical arrow is surjective by Lemma \ref{lemmepetitdegre} (as~$n\leq 2$).
It follows from~\eqref{restrictionsheafification} that the morphism $\Hr_\nr^n(\Pb_\Rb^{d+1},\Zb/2)\to\Hr_\nr^n(X,\Zb/2)$ is surjective. The morphism~$\Hr^n_{\nr}(\Spec(\Rb),\Zb/2)\to\Hr^n_\nr(\Pb_\Rb^{d+1},\Zb/2)$ given by pullback from the point is surjective (see \cite[Proposition 1.2]{colliot-theleneVarietesUnirationnellesNon1989}, and use the comparison between \'etale and $G$\nobreakdash-equivariant cohomology \cite{coxEtaleHomotopyType1979}), hence so is~${\Hr^n_{\nr}(\Spec(\Rb),\Zb/2)\to\Hr_\nr^n(X,\Zb/2)}$. The class $\langle 1,1\rangle^{\otimes n}\in \Hr^0(X,\Ibf^n)$ is sent by the top right horizontal arrow of \eqref{signsignev} to the image in~$\Hr_\nr^n(X,\Zb/2)$ of the nonzero element of~$\Hr^n_{\nr}(\Spec(\Rb),\Zb/2)$, hence to a generator of~$\Hr_\nr^n(X,\Zb/2)$. We deduce from \eqref{signsignev} that $\Hr^0(X,\Ibf^n)$ is generated by~$\Hr^0(X,\Ibf^{n+1})$ and $\langle 1,1\rangle^{\otimes n}$. As~$\sign_{n,X}^0(\langle 1,1\rangle^{\otimes n})=1$, it follows from \eqref{signsignev} again that the image of $\sign_{n,X}^0$ is generated by~$1\in \Hr^0(X(\Rb),\Zb)$ and by twice the image of~$\sign_{n+1,X}^0$. The induction hypothesis completes the proof.
\end{proof}

Let us recall a few facts for use in the proof of the next proposition. Let $X$ be a connected smooth variety over $\Rb$. The group~$\Wr_{\nr}(X):=\Hr^0(X,\Wbf)$ can be identified with the subgroup of $\Wr(\Rb(X))$ consisting of those forms whose residues along all integral divisors of~$X$ vanish (see~\cite[Lemma 4.2 and Theorem~6.1]{balmerGillePaninWalterGersten2002}). It is a birational invariant of smooth proper varieties (combine \cite[Proposition~2.1.8\,e)]{colliot-theleneBirationalInvariantsPurity1995} and \cite[Theorem~6.1]{balmerGillePaninWalterGersten2002}). For $n\geq 0$, one has $\Hr^0(X,\Ibf^n)=\Wr_{\nr}(X)\cap \Ir^n(\Rb(X))$ (this follows from \cite[Theorem 2.3]{morelMilnorConjectureQuadratic2005}, see also \cite[Theorem 2.25]{lerbetImageHigherSignature2026}). As a consequence, the group $\Hr^0(X,\Ibf^n)$ is a birational invariant of smooth proper varieties.

\begin{prop}
\label{propPfistercubic}
Let $X\subset \Pb^{d+1}_{\Rb}$ be a smooth cubic hypersurface of dimension~$d\geq 5$. 
Then $\sign^0_{d-1,X}:\Hr^0(X,\Ibf^{d-1})\to\Hr^0(X(\Rb),\Zb)$ is surjective. As a consequence, the cokernel of the signature morphism $\Wr_{\nr}(X)\to \Hr^0(X(\Rb),\Zb)$ is killed by~$2^{d-1}$.
\end{prop}

\begin{proof}
Since $d\geq 5$, one can use \cite[Proposition 1]{debarreSurIntersectionsCompletes2000} to find a $2$-dimensional linear subspace~$\Pi\subset X$. Let $\widetilde{\Pb^{d+1}_{\Rb}}\to \Pb^{d+1}_{\Rb}$ and $\widetilde{X}\to X$ be the blowups along $\Pi$, so $\widetilde{X}$ is the strict transform of $X$ in $\widetilde{\Pb^{d+1}_{\Rb}}$. Let $\pi':\widetilde{\Pb^{d+1}_{\Rb}}\to\Pb^{d-2}_{\Rb}$ and $\pi:\widetilde{X}\to\Pb^{d-2}_{\Rb}$ be the projections from $\Pi$. As $\pi'$ is a $\Pb^3$-bundle and $\Oc_{\widetilde{\Pb^{d+1}_{\Rb}}}(\widetilde{X})$ has degree~$2$ on the fibers of $\pi'$, each fiber of $\pi$ is isomorphic to a $\Pb^3$, or to a possibly singular quadric surface.

By birational invariance of $\Hr^0(X,\Ibf^{d-1})$, it suffices to show that 
$$\sign^0_{d-1,\widetilde{X}}:\Hr^0(\widetilde{X},\Ibf^{d-1})\to\Hr^0(\widetilde{X}(\Rb),\Zb)$$
is surjective. Fix a connected component $K\subset \widetilde{X}(\Rb)$. By definition of~$\sign^0_{d-1,\widetilde{X}}$, we must find ${q_K\in \Hr^0(\widetilde{X},\Ibf^{d-1})}$ whose signature, when viewed as an element of~$\Wr_{\nr}(\widetilde{X})$, is equal to $2^{d-1}$ on $K$ and to $0$ on the other connected components of~$\widetilde{X}(\Rb)$. 

We first construct $q_K$. The above description of the fibers of $\pi$ implies that the real loci of the real fibers of $\pi$ are empty or connected. The compact subsets~$\pi(K)$ and~${\pi(\widetilde{X}(\Rb)\setminus K)}$ of~$\Pb^{d-2}(\Rb)$ are therefore disjoint. Applying the Stone--Weierstrass theorem to an affine open neighborhood~$U\subset\Pb^{d-2}_{\Rb}$ of~$\Pb^{d-2}(\Rb)$ yields a rational function~${f}$ on~$\Pb^{d-2}_{\Rb}$ that is regular along~$\Pb^{d-2}(\Rb)$, negative on~$\pi(K)$, and positive on~${\pi(\widetilde{X}(\Rb)\setminus K)}$. Define~$q_K$ to be the Pfister form $\langle 1,-f\rangle\otimes\langle 1,1\rangle^{\otimes d-2}$. The form~$q_K$ belongs to $\Ir^{d-1}(\Rb(\widetilde{X}))$ and has the required signature by our choice of~$f$. It remains to prove that~$q_K$ belongs to the subgroup~${\Wr_{\nr}(\widetilde{X})}$.
 
Fix an integral divisor $D\subset \widetilde{X}$. Set~$F:=\Rb(D)$. Our goal is to show that the residue of~$q_K$ along~$D$ vanishes. If the order of vanishing $\ord_D(f)$ of $f\in\Rb(\widetilde{X})$ along~$D$ is even, this is automatic (see e.g.\,\cite[Remark 19.12]{elmanKarpenkoMerkurjevQuadraticForms2008}). We may therefore assume that~$\ord_D(f)$ is odd. Then the residue is a multiple of~$\langle 1,1\rangle^{\otimes d-2}$ (again by \cite[Remark 19.12]{elmanKarpenkoMerkurjevQuadraticForms2008}). We will show that this form vanishes in~$\Wr(F)$. 

Set $D':=\pi(D)$ and $F':=\Rb(D')$. As $\ord_D(f)$ is nonzero, the inclusion ${D'\subset \Pb^{d-2}_{\Rb}}$ is strict, so we have $\dim(D')\leq d-3$. It follows from \cite[Theorem~31.22 and Corollary~35.29]{elmanKarpenkoMerkurjevQuadraticForms2008} that any class in~$\Ir^{d-2}(F')$ whose signatures with respect to all field orderings of $F'$ vanish must be the trivial class. If $F'$ admits no field ordering, this implies that $\Ir^{d-2}(F')=0$, hence that $\langle 1,1\rangle^{\otimes d-2}$ already vanishes in $\Wr(F')$. We henceforth assume that $F'$ admits at least one field ordering.

Let $Q$ be the generic fiber of $\pi|_D:D\to D'$. It is an irreducible component of a fiber of $\pi$, so either it is a smooth quadric surface or one has $Q(F')\neq\varnothing$. We claim that~$Q$ has no point in any real closed extension of $F'$. As $F'$ admits a field ordering, it follows that $Q$ must be a smooth quadric surface. To prove the claim, note that~$f$ is invertible along~$\widetilde{X}(\Rb)$, so $D(\Rb)=\varnothing$. We deduce from the Tarski--Seidenberg transfer principle (see \cite[Proposition 5.2.3]{BochnakCosteRoy1998}) that~$D$ has no point in any real closed extension of $\Rb$. Hence $Q$ has no point in any real closed extension of~$F'$.
 
 Let $a,b,c\in (F')^\times$ be such that the form $\langle 1,a,b,c\rangle$ is an equation of $Q$. Since~$Q$ has no point in any real closed extension of $F'$, the elements~$a,b,c$ are positive with respect to all the field orderings of $F'$. The forms~$\langle1,a\rangle\otimes\langle1,b\rangle\otimes\langle1,c\rangle\otimes\langle1,1\rangle^{\otimes d-5}$ and~$\langle1,1\rangle^{\otimes d-2}$ therefore have identical signatures. Since both belong to $\Ir^{d-2}(F')$, they are equal in $\Wr(F')$. As the Pfister form $\langle1,a\rangle\otimes\langle1,b\rangle\otimes\langle1,c\rangle$ contains~$\langle 1,a,b,c\rangle$ as a subform, it is isotropic over $F'(Q)=F$, hence hyperbolic over $F$ (see \cite[Corollary 9.10]{elmanKarpenkoMerkurjevQuadraticForms2008}). Thus $\langle1,1\rangle^{\otimes d-2}$ vanishes in $\Wr(F)$, as required.
\end{proof}

\subsubsection{Products with \texorpdfstring{$\Pb^1$}{the projective line}}

The following lemma will enable us to carry exponent bounds under $\Pb^1$-suspension in the motivic sense. It is possible to deduce it from general principles of stable motivic homotopy theory; we prefer to give a direct proof to keep the paper more self-contained. 

\begin{lem}
\label{lem:P1_suspension_image_quadratic_real_cycle_class_map}
Let $c,d,n\geq 0$ be integers. Let $X$ be a smooth variety of dimension~$d$ over $\Rb$ and let $\Lc$ be a line bundle on $X$. Let $p:\Pb_X^1\to X$ be the structural morphism. Then $\sign_{n,X}^{c}$ is a retract of~$\sign_{n+1,\Pb_X^1}^{c+1}$, in the sense that there exists a diagram
\[\begin{tikzcd}
	{\Hr^{c}(X,\Ibf^{n}(\Lc))} & {\Hr^{c+1}(\Pb_X^1,\Ibf^{n+1}(p^*\Lc))} \\
	{\Hr^{c}(X(\Rb),\Zb(\Lc))} & {\Hr^{c+1}(\Pb_X^1(\Rb),\Zb(p^*\Lc))}
	\arrow["\alpha"', shift right=2, from=1-1, to=1-2]
	\arrow["{\sign_{n,X}^{c}}"', from=1-1, to=2-1]
	\arrow["\beta"', from=1-2, to=1-1]
	\arrow["{\sign_{n+1,\Pb_X^1}^{c+1}}", from=1-2, to=2-2]
	\arrow["{\alpha'}", shift left=2, from=2-1, to=2-2]
	\arrow["{\beta'}", from=2-2, to=2-1]
\end{tikzcd}\]
whose inner and outer squares commute and such that ${\beta\circ\alpha=\Id}$ and $\beta'\circ\alpha'=\Id$.
\end{lem}

\begin{proof}
Let $p:\Pb_X^1\to X$ and $q:\Pb_X^1\to \Pb^1_\Rb$ be the projections.
Let $i:X\to\Pb_X^1$ be the closed immersion induced by the point at infinity of $\Pb^1_{\Rb}$, and let $j:\Abb_X^1\hookrightarrow\Pb_X^1$ be the complementary open immersion. We also denote by $r:\Abb_X^1\to X$ the projection. 

Consider the localization long exact sequence in $\Ibf$-cohomology
{\small
\begin{equation*}
\cdots \to \Hr^c(X,\Ibf^n(\omega_{X/\Pb_X^1}\otimes\Lc))
\xrightarrow{i_*}
\Hr^{c+1}(\Pb_X^1,\Ibf^{n+1}(p^*\Lc))
\xrightarrow{j^*}
\Hr^{c+1}(\Abb_X^1,\Ibf^{n+1}(r^*\Lc))
\to\cdots
\end{equation*}
}
stemming from \cite[Theorem 2.41]{hornbostelRealCycleClass2021}. One computes that
$$\omega_{X/\Pb_X^1}=\omega_{X/\Rb}\otimes i^*\omega_{\Pb_X^1/\Rb}^{-1}=\omega_{X/\Rb}\otimes
i^*(p^*\omega_{X/\Rb}\otimes q^*\omega_{\Pb_\Rb^1/\Rb})^{-1}\simeq\Oc_X,$$
where the last isomorphism holds as $p\circ i=\Id$ and $q\circ i$ is constant. In addition, for all $m\geq 0$, since $r^*=j^*\circ p^*$ and $r^*:\Hr^m(X,\Ibf^{n+1}(\Lc))\to\Hr^m(\Abb_X^1,\Ibf^{n+1}(r^*\Lc))$ is an isomorphism by \cite[Théorème 11.2.9]{faselGroupesChowWitt2008}, the morphism $\sigma:=p^*\circ(r^*)^{-1}$ is a section of~$j^*:\Hr^m(\Pb_X^1,\Ibf^{n+1}(p^*\Lc))\to\Hr^m(\Abb_X^1,\Ibf^{n+1}(r^*\Lc))$ which is therefore surjective. It follows that the above localization exact sequence induces a short exact sequence
\begin{equation*}
0\to\Hr^c(X,\Ibf^n(\Lc))
\xrightarrow{i_*}
\Hr^{c+1}(\Pb_X^1,\Ibf^{n+1}(p^*\Lc))
\xrightarrow{j^*}
\Hr^{c+1}(\Abb_X^1,\Ibf^{n+1}(r^*\Lc))
\to0
\end{equation*}
which is split by $\sigma$. An entirely similar argument based on the localization exact sequence, on the Thom isomorphism, and on homotopy invariance for cohomology with locally constant coefficients yields a short exact sequence
\begin{equation*}
0\to\Hr^c(X(\Rb),\Zb(\Lc))
\xrightarrow{i(\Rb)_*}
\Hr^{c+1}(\Pb_X^1(\Rb),\Zb(p^*\Lc))
\xrightarrow{j(\Rb)^*}
\Hr^{c+1}(\Abb_X^1(\Rb),\Zb(r^*\Lc))
\to0
\end{equation*}
which is split by $\sigma':=p(\Rb)^*\circ (r(\Rb)^*)^{-1}$. 

Set $\alpha:=i_*$ and $\alpha':=i(\Rb)_*$ and define $\beta:=\Id-(\sigma\circ j^*)$ and ${\beta':=\Id-(\sigma\circ j(\Rb)^*)}$, so $\beta$ and $\beta'$ are the retractions of $\alpha$ and~$\alpha'$ induced by $\sigma$ and $\sigma'$. In the diagram
\[{
\small\begin{tikzcd}[row sep=1.5em, column sep=2em]
	0 & {\Hr^c(X,\Ibf^n(\Lc))}
	  & {\Hr^{c+1}(\Pb_X^1,\Ibf^{n+1}(p^*\Lc))}
	  & {\Hr^{c+1}(\Abb_X^1,\Ibf^{n+1}(r^*\Lc))} & 0 \\
	0 & {\Hr^c(X(\Rb),\Zb(\Lc))}
	  & {\Hr^{c+1}(\Pb_X^1(\Rb),\Zb(p^*\Lc))}
	  & {\Hr^{c+1}(\Abb_X^1(\Rb),\Zb(r^*\Lc))} & 0,
	\arrow[from=1-1, to=1-2]
	\arrow["\alpha"',shift right=2, from=1-2, to=1-3]
	\arrow["\beta"',from=1-3, to=1-2]
	\arrow["{\sign_{n,X}^c}"', from=1-2, to=2-2]
	\arrow["{j^*}"',shift right=2, from=1-3, to=1-4]
	\arrow["\sigma"',from=1-4, to=1-3]
	\arrow["{\sign_{n+1,\Pb_X^1}^{c+1}}"', from=1-3, to=2-3]
	\arrow[from=1-4, to=1-5]
	\arrow["{\sign_{n+1,\Abb_X^1}^{c+1}}"', from=1-4, to=2-4]
	\arrow[from=2-1, to=2-2]
	\arrow["{\alpha'}", shift left=2, from=2-2, to=2-3]
	\arrow["\beta'",from=2-3, to=2-2]
	\arrow["{j(\Rb)^*}",shift left=2, from=2-3, to=2-4]
	\arrow["\sigma'",from=2-4, to=2-3]
	\arrow[from=2-4, to=2-5]
\end{tikzcd}}
\]
the left outer square and the right inner and outer squares commute by functoriality of higher signature morphisms. A diagram chase then shows that the left inner square also commutes. This completes the proof of the lemma.
\end{proof}

\subsubsection{Higher signature morphisms with small image}

\begin{prop}
\label{propdeYaX}
Fix $0\leq c<d$ and $0\leq n< d$. Set $m:=\max(0,n-c)$. Let~$Y$ be a smooth variety of dimension $d-c$ over $\Rb$. Define $X:=Y\times(\Pb_\Rb^1)^{c}$. If the cokernel of $\sign_{m,Y}^0:\Hr^0(Y,\Ibf^m)\to \Hr^0(Y(\Rb),\Zb)$ is not killed by $2^{d-c-m-1}$, then the cokernel of $\sign_{n,X}^c:\Hr^c(X,\Ibf^n)\to \Hr^c(X(\Rb),\Zb)$  is not killed by~$2^{d-n-1}$.
\end{prop}

\begin{proof}
Suppose first that $n\geq c$. Applying $c$ times Lemma \ref{lem:P1_suspension_image_quadratic_real_cycle_class_map} shows that $\sign_{n-c,Y}^0$ is a retract of $\sign_{n,X}^c$. Since the cokernel of $\sign_{n-c,Y}^0$ is not killed by $2^{d-n-1}$, we deduce that the cokernel of $\sign_{n,X}^c$ is not killed by $2^{d-n-1}$, as required.

Suppose now that $n<c$. Iterating Lemma~\ref{lem:P1_suspension_image_quadratic_real_cycle_class_map} yields a surjective~morphism
$$\theta:\Hr^c(X(\Rb),\Zb)\to\Hr^0(Y(\Rb),\Zb)$$
satisfying $\theta\bigl(\Im(\sign_{c,X}^c)\bigr)=\Im(\sign_{0,Y}^0)$ (to be precise, the morphism $\theta$ is the composition of $c$ of the morphisms $\beta'$ from Lemma~\ref{lem:P1_suspension_image_quadratic_real_cycle_class_map}). By \cite[(2-4)]{lerbetImageHigherSignature2026}, the inclusion $\Ibf^c\subseteq\Ibf^n$ induces a commutative~square
\begin{equation}
\label{signsignnc}
\begin{tikzcd}[row sep=small]
	{\Hr^c(X,\Ibf^c)} & {\Hr^c(X,\Ibf^n)} \\
	{\Hr^c(X(\Rb),\Zb)} & {\Hr^c(X(\Rb),\Zb)}
	\arrow[two heads, from=1-1, to=1-2]
	\arrow["{\sign_{c,X}^c}"', from=1-1, to=2-1]
	\arrow["{\,\sign_{n,X}^c}", from=1-2, to=2-2]
	\arrow["{\cdot 2^{c-n}}"', from=2-1, to=2-2]
\end{tikzcd}
\end{equation}
whose top horizontal arrow is surjective (see e.g.\,\cite[Lemma~2.6]{lerbetWittGroupsSmooth2026}). It follows from diagram~\eqref{signsignnc} and from the equality $\theta\bigl(\Im(\sign_{c,X}^c)\bigr)=\Im(\sign_{0,Y}^0)$ that
\begin{equation}
\label{thetasign}
\theta\bigl(\Im(\sign_{n,X}^c)\bigr)=2^{c-n}\,\theta\bigl(\Im(\sign_{c,X}^c)\bigr)=2^{c-n}\Im(\sign_{0,Y}^0).
\end{equation}

Since the cokernel of ${\sign_{0,Y}^0:\Wr_{\nr}(Y)\to\Hr^0(Y(\Rb),\Zb)}$ is not killed by $2^{d-c-1}$ and~$\Hr^0(Y(\Rb),\Zb)$ is torsion free, the cokernel of ${2^{c-n}\,\sign_{0,Y}^0}$ is not killed by $2^{d-n-1}$.
It follows from \eqref{thetasign} that $\Hr^0(Y(\Rb),\Zb)/\theta\bigl(\Im(\sign_{n,X}^c)\bigr)$ is not killed by $2^{d-n-1}$. As~$\theta$ is surjective, we deduce that the cokernel of~$\sign_{n,X}^c$ is not killed by~$2^{d-n-1}$.
\end{proof}

We finally prove Theorem \ref{thexpcokeroptimalintro} from the introduction.

\begin{thm}
\label{thexpcokeroptimal}
Fix $0\leq c\leq d$ and $0\leq n< d$. There exists a smooth variety $X$ of dimension~$d$ over~$\Rb$ such that the cokernel of 
$$\sign_{n,X}^c:\Hr^c(X,\Ibf^n)\to \Hr^c(X(\Rb),\Zb)$$ 
is not killed by~$2^{d-n-1}$. If~$d-c\leq 4$, one can moreover choose $X$ to be projective.
\end{thm}

\begin{proof}
If $c=d$, take $X:=(\Pb^1_{\Rb})^c$. As $\sign_{c,X}^c$ is surjective by Corollary~\ref{cor:integral_cycle_class_map_split_surjective}, diagram~\eqref{signsignnc} implies that the cokernel of $\sign_{n,X}^c$ is not killed by $2^{d-n-1}$.

Assume now that $c<d$. To prove the first assertion, apply Proposition \ref{propdeYaX} to~$Y:=\Gb_{m,\Rb}^{d-c}$ (this is legitimate by Proposition \ref{propmultiplicative}). To prove the second assertion, apply it to a very general smooth projective cubic hypersurface $Y$ of dimension~${d-c}$ over~$\Rb$ such that $Y(\Rb)$ is disconnected (use Proposition \ref{prop:cokernel_cubic}).
\end{proof}

\subsubsection{Applications to Chow--Witt groups}
\label{parChowWitt}

Let $X$ be a smooth variety over $\Rb$. Let~$\Lc$ be a line bundle on $X$.
The real cycle class map ${\widetilde{\cl}_{\Rb}:\widetilde{\CH}^c(X,\Lc)\to \Hr^c(X(\Rb),\Zb(\Lc))}$ for Chow--Witt groups (see \cite{faselGroupesChowWitt2008} for Chow--Witt groups and \cite[\S 3.D]{hornbostelRealCycleClass2021} and \cite[\S 2.5]{lerbetImageHigherSignature2026} for the real cycle class map in this context) factors as 
$$\widetilde{\CH}^c(X,\Lc)\twoheadrightarrow\Hr^c(X,\Ibf^c(\Lc))\xrightarrow{\sign_{c,X}^c} \Hr^c(X(\Rb),\Zb(\Lc)),$$
where the first arrow is surjective by \cite[Lemma 2.23]{lerbetImageHigherSignature2026}. Theorem~\ref{thexpcoker} applied with~$c=n$ therefore has the following consequence.

\begin{cor}
Let $X$ be a smooth variety of dimension $d$ over $\Rb$. Let $\Lc$ be a line bundle on $X$. For all $0\leq c\leq d$, the cokernel of the real cycle class map~${\widetilde{\cl}_{\Rb}:\widetilde{\CH}^c(X,\Lc)\to \Hr^c(X(\Rb),\Zb(\Lc))}$ is killed by~$2^{d-c}$.
\end{cor}

Similarly, taking $n=c$ in Theorem \ref{thexpcokeroptimal} yields the following corollary.

\begin{cor}
Fix $0\leq c<d$. There exists a smooth variety~$X$ of dimension $d$ over $\Rb$ such that the cokernel of the real cycle class map~${\widetilde{\cl}_{\Rb}:\widetilde{\CH}^c(X)\to \Hr^c(X(\Rb),\Zb)}$ is not killed by~$2^{d-c-1}$. If $d-c\leq 4$, one can moreover choose $X$ to be projective.
\end{cor}

 {\footnotesize
\bibliographystyle{amsalpha}
\bibliography{references}}

\newcommand{\etalchar}[1]{$^{#1}$}
\providecommand{\bysame}{\leavevmode\hbox to3em{\hrulefill}\thinspace}
\providecommand{\MR}{\relax\ifhmode\unskip\space\fi MR }
\providecommand{\MRhref}[2]{%
  \href{http://www.ams.org/mathscinet-getitem?mr=#1}{#2}
}
\providecommand{\href}[2]{#2}
\begin{thebibliography}{BGPW02}

\bibitem[Bal05]{Balmer2005}
Paul Balmer, \emph{{W}itt groups}, Handbook of K-Theory, vol.~2, Springer,
  2005, pp.~539--576.

\bibitem[BCD{\etalchar{+}}25]{bachmannMilnorWittMotives2025}
Tom Bachmann, Baptiste Calm{\`e}s, Fr{\'e}d{\'e}ric D{\'e}glise, Jean Fasel,
  and Paul~Arne {\O}stv{\ae}r, \emph{Milnor--{{Witt}} motives}, vol. 311,
  Memoirs of the AMS, no. 1572, AMS, 2025.

\bibitem[BCR98]{BochnakCosteRoy1998}
Jacek Bochnak, Michel Coste, and Marie-Fran{\c c}oise Roy, \emph{Real algebraic
  geometry}, Ergeb. Math. Grenzgeb., vol.~36, Springer-Verlag, 1998.

\bibitem[Ben19]{benoistPeriodindexProblemReal2019}
Olivier Benoist, \emph{The period-index problem for real surfaces}, Publ. math.
  IH\'{E}S \textbf{130} (2019), no.~1, 63--110.

\bibitem[BGPW02]{balmerGillePaninWalterGersten2002}
Paul Balmer, Stefan Gille, Ivan Panin, and Charles Walter, \emph{The {Gersten}
  conjecture for {Witt} groups in the equicharacteristic case}, Doc. Math.
  \textbf{7} (2002), 203--217.

\bibitem[BL04]{birkenhakeLangeComplexAbelianVarieties2004}
Christina Birkenhake and Herbert Lange, \emph{Complex abelian varieties}, 2nd
  ed., Grundlehren Math. Wiss., vol. 302, Springer, 2004.

\bibitem[BO20]{benoistOttemFailureIHC2020}
Olivier Benoist and John~Christian Ottem, \emph{Failure of the integral {Hodge}
  conjecture for threefolds of {Kodaira} dimension zero}, Comment. Math. Helv.
  \textbf{95} (2020), no.~1, 27--35.

\bibitem[Br{\"o}80]{brockerReelleDivisoren1980}
Ludwig Br{\"o}cker, \emph{Reelle {D}ivisoren}, Archiv der Mathematik
  \textbf{35} (1980), 140--143.

\bibitem[BW02]{balmerWalterGerstenWitt2002}
Paul Balmer and Charles Walter, \emph{A {Gersten--Witt} spectral sequence for
  regular schemes}, Ann. Sci. \'ENS \textbf{35} (2002), no.~1, 127--152.

\bibitem[BW20]{benoistIntegralHodgeConjecture2020}
Olivier Benoist and Olivier Wittenberg, \emph{On the integral {{Hodge}}
  conjecture for real varieties, {{I}}}, Invent. math. \textbf{222} (2020),
  no.~1, 1--77.

\bibitem[BW24]{BWWu2024}
\bysame, \emph{{T}he {W}u relations in real algebraic geometry},
  \href{http://arxiv.org/abs/2405.11603}{arXiv:2405.11603}, to appear in Duke
  Math. J., 2024.

\bibitem[CO89]{colliot-theleneVarietesUnirationnellesNon1989}
Jean-Louis {Colliot-Th{\'e}l{\`e}ne} and Manuel Ojanguren,
  \emph{{Vari{\'e}t{\'e}s unirationnelles non rationnelles: au-del{\`a} de
  l'exemple d'Artin et Mumford}}, Invent. math. \textbf{97} (1989), no.~1,
  141--158.

\bibitem[Cox79]{coxEtaleHomotopyType1979}
David~A. Cox, \emph{The \'etale homotopy type of varieties over
  {{$\mathbb{R}$}}}, Proc. AMS \textbf{76} (1979), no.~1, 17--22.

\bibitem[CP90]{colliot-theleneRealComponentsAlgebraic1990}
Jean-Louis {Colliot-Th{\'e}l{\`e}ne} and Raman Parimala, \emph{Real components
  of algebraic varieties and étale cohomology}, Invent. math. \textbf{101}
  (1990), no.~1, 81--99.

\bibitem[CS96]{colliot-theleneZerocyclesCohomologyReal1996}
Jean-Louis {Colliot-Th{\'e}l{\`e}ne} and Claus Scheiderer, \emph{Zero-cycles
  and cohomology on real algebraic varieties}, Topology \textbf{35} (1996),
  no.~2, 533--559.

\bibitem[CT95]{colliot-theleneBirationalInvariantsPurity1995}
Jean-Louis Colliot-Th{\'e}l{\`e}ne, \emph{Birational invariants, purity and the
  {Gersten} conjecture}, $K$-Theory and Algebraic Geometry: Connections with
  Quadratic Forms and Division Algebras, Proc. Sympos. Pure Math., vol.~58,
  AMS, 1995, pp.~1--64.

\bibitem[CTV12]{colliot-theleneCohomologieNonRamifiee2012}
Jean-Louis Colliot-Thélène and Claire Voisin, \emph{Cohomologie non ramifiée
  et conjecture de {{Hodge}} entière}, Duke Math. J. \textbf{161} (2012),
  no.~5, 735--801.

\bibitem[DIK00]{degtyarevRealEnriquesSurfaces2000}
Alexander Degtyarev, Ilia Itenberg, and Viatcheslav Kharlamov, \emph{Real
  {{Enriques surfaces}}}, Lecture {{Notes}} in {{Math.}}, vol. 1746, Springer,
  2000.

\bibitem[DK96]{degtyarevHalvesRealEnriques1996}
Alexander Degtyarev and Viatcheslav Kharlamov, \emph{Halves of a real
  {{Enriques}} surface}, Comment. Math. Helv. \textbf{71} (1996), no.~1,
  628--663.

\bibitem[DM00]{debarreSurIntersectionsCompletes2000}
Olivier Debarre and Laurent Manivel, \emph{Sur les intersections compl{\`e}tes
  r{\'e}elles}, CRAS Paris S{\'e}r. I Math. \textbf{331} (2000), no.~11,
  887--892.

\bibitem[EKM08]{elmanKarpenkoMerkurjevQuadraticForms2008}
Richard Elman, Nikita Karpenko, and Alexander Merkurjev, \emph{The algebraic
  and geometric theory of quadratic forms}, AMS Colloquium Publications,
  vol.~56, AMS, 2008.

\bibitem[Fas08]{faselGroupesChowWitt2008}
Jean Fasel, \emph{{Groupes de Chow--Witt}}, {M{\'e}moires de la SMF}, no. 113,
  SMF, Paris, 2008.

\bibitem[FT19]{fu2CyclesHypersurfacesCubiques2019}
Lie Fu and Zhiyu Tian, \emph{$2$-cycles sur les hypersurfaces cubiques de
  dimension $5$}, Math. Z. \textbf{293} (2019), no.~1--2, 661--676.

\bibitem[GH81]{grossHarrisRealAlgebraicCurves1981}
Benedict~H. Gross and Joe Harris, \emph{Real algebraic curves}, Ann. Sci. \'ENS
  \textbf{14} (1981), no.~2, 157--182.

\bibitem[Gro57]{tohoku}
Alexandre Grothendieck, \emph{Sur quelques points d'alg\`ebre homologique},
  T\^ohoku Math. J. (2) \textbf{9} (1957), 119--221.

\bibitem[Hen25]{hennigRealCycleClass2025}
Jan Hennig, \emph{The real cycle class isomorphism for linear schemes},
  \href{http://arxiv.org/abs/2511.02549}{arXiv:2511.02549}, 2025.

\bibitem[Hor25]{hornbostelFewComputationsReal2025}
Jens Hornbostel, \emph{A few computations about the real cycle class map in low
  dimensions}, Q. J. Math. \textbf{76} (2025), no.~1, 185--194.

\bibitem[Huy23]{huybrechtsGeometryCubicHypersurfaces2023}
Daniel Huybrechts, \emph{The geometry of cubic hypersurfaces}, Cambridge
  {{Studies}} in {{Advanced Math.}}, vol. 206, Cambridge University Press,
  2023.

\bibitem[HWXZ21]{hornbostelRealCycleClass2021}
Jens Hornbostel, Matthias Wendt, Heng Xie, and Marcus Zibrowius, \emph{The real
  cycle class map}, Ann. $K$-Theory \textbf{6} (2021), no.~2, 239--317.

\bibitem[Jac17]{jacobsonRealCohomologyPowers2017}
Jeremy~A. Jacobson, \emph{Real cohomology and the powers of the fundamental
  ideal in the {{Witt}} ring}, Ann. $K$-Theory \textbf{2} (2017), no.~3,
  357--385.

\bibitem[Kne76]{Knebusch1976II}
Manfred Knebusch, \emph{{O}n algebraic curves over real closed fields. {II}},
  Math. Z. \textbf{151} (1976), no.~2, 189--205.

\bibitem[Kne77]{Knebusch1977}
\bysame, \emph{{S}ymmetric bilinear forms over algebraic varieties}, Conference
  on Quadratic Forms, Queen's Papers in Pure and Applied Mathematics, vol.~46,
  Queen's University, Kingston, Ontario, 1977, pp.~103--283.

\bibitem[KR26]{KharlamovRasdeaconu2026}
Viatcheslav Kharlamov and Rare{\c s} R{\u a}sdeaconu, \emph{{K}alinin
  effectivity and wonderful compactifications},
  \href{https://arxiv.org/abs/2603.10372}{arXiv:2603.10372}, 2026.

\bibitem[Kra91]{krasnovCharacteristicClasses1991}
Vyacheslav~A. Krasnov, \emph{Characteristic classes of vector bundles on a real
  algebraic variety}, Izv. Akad. Nauk SSSR Ser. Mat. \textbf{55} (1991), no.~4,
  716--746.

\bibitem[Kra94]{krasnovEquivariantGrothendieckCohomology1994}
\bysame, \emph{On the equivariant {Grothendieck} cohomology of a real algebraic
  variety and its application}, Izv. Ross. Akad. Nauk Ser. Mat. \textbf{58}
  (1994), no.~3, 36--52.

\bibitem[Laz04]{lazarsfeldPositivityII}
Robert Lazarsfeld, \emph{Positivity in algebraic geometry {II}}, Ergeb. Math.
  Grenzgeb., vol.~49, Springer-Verlag, 2004.

\bibitem[Ler26a]{lerbetImageHigherSignature2026}
Samuel Lerbet, \emph{On the image of higher signature maps}, Ann. $K$-Theory
  \textbf{11} (2026), no.~2, 171--212.

\bibitem[Ler26b]{lerbetWittGroupsSmooth2026}
\bysame, \emph{Witt groups of smooth real curves and surfaces},
  \href{http://arxiv.org/abs/2608.18599}{arXiv:2608.18599}, 2026.

\bibitem[Lew93]{lewisCylinderHomomorphisms1993}
James~D. Lewis, \emph{Cylinder homomorphisms and {C}how groups}, Math. Nachr.
  \textbf{160} (1993), 205--221.

\bibitem[Mah82]{Mahe1982}
Louis Mah{\'e}, \emph{{S}ignatures et composantes connexes}, Math. Ann.
  \textbf{260} (1982), no.~2, 191--210.

\bibitem[Mon97]{Monnier1997}
Jean-Philippe Monnier, \emph{{A}bout the image of the total signature map in
  the two-dimensional case}, manuscripta math. \textbf{93} (1997), no.~1,
  143--161.

\bibitem[Mon00]{Monnier2000}
\bysame, \emph{{U}nramified cohomology and quadratic forms}, Math. Z.
  \textbf{235} (2000), no.~3, 455--478.

\bibitem[Mor05]{morelMilnorConjectureQuadratic2005}
Fabien Morel, \emph{Milnor's conjecture on quadratic forms and mod 2 motivic
  complexes}, Rend. Sem. Mat. Univ. Padova \textbf{114} (2005), 63--101.

\bibitem[MvH98]{MangolteVanHamel1998}
Fr{\'e}d{\'e}ric Mangolte and Joost van Hamel, \emph{{A}lgebraic cycles and
  topology of real {E}nriques surfaces}, Compos. Math. \textbf{110} (1998),
  no.~2, 215--238.

\bibitem[OVV07]{orlovExactSequenceKM22007}
Dmitri Orlov, Alexander Vishik, and Vladimir Voevodsky, \emph{An exact sequence
  for {$K_*^M/2$} with applications to quadratic forms}, Ann. Math.
  \textbf{165} (2007), no.~1, 1--13.

\bibitem[Par92]{Parimala1992}
Raman Parimala, \emph{{A}lgebraic geometric methods in real algebraic
  geometry}, Real Algebraic Geometry, Lecture Notes in Math., vol. 1524,
  Springer, 1992, pp.~37--51.

\bibitem[Rei72]{reidCompleteIntersectionQuadrics1972}
Miles Reid, \emph{The complete intersection of two or more quadrics}, Ph.D.
  thesis, University of Cambridge, 1972.

\bibitem[Sil82]{silholRealAbelianVarieties1982}
Robert Silhol, \emph{Real abelian varieties and the theory of {C}omessatti},
  Math. Z. \textbf{181} (1982), no.~3, 345--364.

\bibitem[Sil89]{silholRealAlgebraicSurfaces1989}
\bysame, \emph{Real algebraic surfaces}, Lecture Notes in Math., vol. 1392,
  Springer-Verlag, 1989.

\bibitem[vH00a]{vanHamelAlgebraicCyclesTopology2000}
Joost van Hamel, \emph{Algebraic cycles and topology of real algebraic
  varieties}, CWI Tract, vol. 129, Stichting Mathematisch Centrum, Centrum voor
  Wiskunde en Informatica, 2000.

\bibitem[vH00b]{vanHamelTorsionZeroCycles2000}
\bysame, \emph{Torsion zero-cycles and the {Abel--Jacobi} map over the real
  numbers}, The Arithmetic and Geometry of Algebraic Cycles, CRM Proceedings \&
  Lecture Notes, vol.~24, AMS, Providence, RI, 2000, pp.~329--359.

\bibitem[Voi12]{voisinDegree4Unramified2012}
Claire Voisin, \emph{Degree 4 unramified cohomology with finite coefficients
  and torsion codimension 3 cycles}, Geometry and Arithmetic, EMS Series of
  Congress Reports, EMS, 2012, pp.~347--368.

\end{thebibliography}
\end{document}